\documentclass[preprint,review,11pt,authoryear]{elsarticle}

\usepackage{caption}
\usepackage{tikz,graphics,color,fullpage,float,epsf,caption,subcaption}
\usepackage{booktabs}
\usepackage{algpseudocode}
\usepackage[ruled, commentsnumbered]{algorithm2e}
\usepackage{pdflscape}
\usepackage{amsmath}
\usepackage{amsfonts}
\usepackage{amssymb}
\usepackage{mathtools}

\usepackage{amsmath,amssymb}
 \usepackage{amsthm}
 \usepackage{todonotes}
 \usepackage{graphicx}
\usepackage{color}
\definecolor{cornell-red}{RGB}{179,27,27}
\usepackage{hyperref} 
\usepackage{subcaption}
\usepackage{mathtools}
\usepackage{dsfont}
\usepackage{color}
\usepackage{natbib}
    \usepackage{multirow}
 \usepackage[T1]{fontenc}

\usepackage[left=1in,right=1in,top=1in,bottom=1in]{geometry}

\usepackage{hyperref} 
\usepackage[flushleft]{threeparttable}
\newtheorem{thm}{Theorem}

\newtheorem{prop}[thm]{Proposition}

\theoremstyle{definition}

\theoremstyle{remark}

\usepackage{natbib}
\newcommand{\blue}{\textcolor{black}}  
\newcommand{\calR}{\mathcal{R}}  

\newcommand{\st}{\text{s.t. }}  

\newcommand{\nb}{\pmb{n}}  
\newcommand{\nL}{\underline{n}}  
\newcommand{\nU}{\overline{n}}  
\newcommand{\calW}{\mathcal{W}}

  \journal{ArXiv (under journal review)}

\begin{document}

 \begin{frontmatter}


 

 

\title{Fleet Sizing and Allocation for On-demand Last-Mile Transportation Systems}

\author[mymainaddress1]{Karmel S. Shehadeh \corref{cor1}}
\cortext[cor1]{Corresponding author. }
 \ead{kshehadeh[AT]lehigh.edu, kas720[AT]lehigh.edu}
   \author[hwaddress,mymainaddress2]{Hai Wang}
 \author[mymainaddress2]{Peter Zhang}
 \address[mymainaddress1]{Department of Industrial and Systems Engineering, Lehigh University, Bethlehem, PA, USA}
\address[mymainaddress2]{Heinz College of Information Systems and Public Policy, Carnegie Mellon University, Pittsburgh, PA, USA}
\address[hwaddress]{School of Computing and Information Systems, Singapore Management University, Singapore}

\begin{abstract}
The last-mile problem refers to the provision of travel service from the nearest public transportation node to home or other destination. Last-Mile Transportation Systems (LMTS), which have recently emerged, provide on-demand shared transportation. In this paper, we investigate the fleet sizing and allocation problem for the on-demand LMTS. Specifically, we consider the perspective of a last-mile service provider who wants to determine the number of servicing vehicles to allocate to multiple last-mile service regions in a particular city. In each service region, passengers demanding last-mile services arrive in batches, and allocated vehicles deliver passengers to their final destinations. \blue {The passenger demand (i.e., the size of each batch of passengers) is random and hard to predict in advance,} \blue{especially with limited data during the planning process}. The quality of fleet-allocation decisions is a function of vehicle fixed cost plus a weighted sum of passenger's waiting time before boarding a vehicle and in-vehicle riding time. We propose and analyze two models---a stochastic programming model and a distributionally robust optimization model---to solve the problem, assuming known and unknown distribution of the demand, respectively. We conduct extensive numerical experiments to evaluate the models and discuss insights and implications into the optimal fleet sizing and allocation for the on-demand LMTS under demand uncertainty. 

\end{abstract}
\end{frontmatter}

\section{Introduction}

\noindent The last-mile problem refers to the design and provision of travel services from a public transportation node to a passenger's final destination. It is a fundamental practical problem that has attracted intense attention in the past decade for several reasons. First, governments worldwide are under pressure to increase public transport's share of urban trips to reduce road congestion and air pollution. Maybe not surprisingly, urban planners recently recognized that the unavailability of last-mile services is one of the main deterrents to the use of public transport. Second, the aging of populations has increased the demand for such services. Third, more countries impose legal requirements to ensure adequate mobility for particular demographic groups, who are most likely to need last-mile services, such as people with physical disabilities. 
\vspace{1mm}

Any passenger needing on-demand last-mile service may provide advance notice to the last-mile transportation systems (LMTS) of his/her impending arrival at the alighting station and her specific final destination. \blue{Once this information is received, the LMTS assigns the passenger to one of the vehicles in the LMTS fleet, plans the vehicle's route so that it includes a stop at the passenger's destination, estimates the vehicle's departure time, and notifies the passenger accordingly}. \blue{Once all of the passengers assigned to a vehicle are on board, the vehicle executes a delivery route with stops at each passenger's destination and returns to the station to pick up passengers for its next delivery tour.} Many papers address various models and case studies of LMTS. With the high penetration of services such as Uber worldwide, most people are aware of the benefits of on-demand transportation services and request even more specialized forms, including last-mile service (\cite{wang2019ridesourcing}) .

\vspace{1mm}

 In this paper, we investigate the fleet sizing and allocation problem for the on-demand LMTS. Specifically, we consider the perspective of a last-mile service provider who wants to determine the number of servicing vehicles to allocate to multiple service regions in a particular city. In each service region (e.g., an area around a metro station), passengers demanding last-mile services arrive in batches (e.g., as a result of consecutive arrivals of metros or trains), and allocated vehicles deliver passengers to their final destinations, i.e., last-mile stops. The size of each batch of passengers (demand henceforth) is random and hard to predict in advance. The quality of fleet-allocation decisions is a function of vehicle fixed cost (vehicle rental or purchase cost)  plus a weighted sum of passenger waiting time before boarding a vehicle and in-vehicle riding time. 
\vspace{1mm}

\blue{The fleet sizing and allocation problem is challenging, especially since the practical passenger demand is stochastic (i.e., uncertain)}. If we know the exact probability distribution of random demand, we may formulate the problem as a two-stage stochastic programming model. In the first stage, we decide the number of vehicles to allocate to each service region and their routes before the realization of random demand. Then in the second stage, we observe the realizations of stochastic demands and make optimal recourse actions (assigning passengers to vehicles), conditioning on first-stage decisions, and accordingly compute the associated passenger waiting and riding times. Mathematically, the stochastic programming model identifies vehicle allocation and routing decisions that minimize the total cost, comprising the fixed cost of vehicle allocation and the expected weighted sum of waiting time and riding time for passengers, where the expectation is taken with respect to this known probability distribution of random demand. 

\vspace{1mm}

\blue{In reality, it is often notoriously difficult to accurately estimate the exact probability distribution of demand, especially with limited data during the planning stage}. \textcolor{black}{LMTS is a relatively emerging transportation mode, and thus  sufficient data on LMTS operations is not readily available. Even if LMTS companies collect data on their LMTS operations, the data may not be sufficient or have high quality to model the demand distribution. Moreover, it is challenging to obtain LMTS data from companies due to privacy issues, among others.} \blue{If we calibrate a stochastic programming model to a data sample from a biased distribution of random demand, then the resulting biased optimal decisions may demonstrate disappointing out-of-sample performance  (in terms of passenger waiting and riding times) under the true distribution-- a phenomenon known as \textit{optimizer's curse} (see \cite{esfahani2018data} and \cite{smith2006optimizer} for a detailed discussion).} Alternatively, one can construct an ambiguity set (i.e., a family) of all plausible probability distributions compatible with the available limited data or expert knowledge about demand. 
\vspace{1mm}

In this paper, we address the uncertainty of passengers' demand for last-mile service. We propose, analyze, and evaluate the computational and operational performance of two models for the fleet sizing and allocation problem, assuming known and unknown distribution of the demand, respectively. First, a stochastic programming model is proposed to minimize the fixed cost of allocated vehicles and the expectation of a weighted sum of passenger waiting and riding times, under a distributional belief of demand. Second, a distributionally robust model is proposed to minimize the fixed cost of vehicles and the worst-case (i.e., maximum) expectation of passenger waiting time and riding times; we also evaluate the worst-case expectation over an ambiguity set (i.e., a family of all possible distributions of uncertain demand) and characterize the ambiguity set by known mean and support information of demand. We conduct extensive numerical experiments and discuss the insights and implications by examining trade-offs between total cost, fleet size, and passenger waiting and riding times. 
\vspace{1mm}

The proposed model generalizes the single-service region LMTS routing and scheduling formulation in literature by (1) considering multiple service regions, (2) considering fleet sizing and allocation decisions, and (3) incorporating the uncertainty of passenger demand. To the best of our knowledge, and according to our literature review in Section~\ref{sec2:Lit}, this is the first paper that provides a theoretical and computational analysis of stochastic optimization models for the fleet sizing and allocation problem for LMTS.
\vspace{1mm}

 The reminder of the paper is structured as follows. In Section~\ref{sec2:Lit}, we review the relevant literature. In Section~\ref{Sec:formulation}, we formally define the problem and propose two model formulations. In Section \ref{sec:CompResults}, we present computational results and discuss managerial implications. Finally, we draw conclusions and discuss future directions in Section~\ref{sec:Conclusion}.

\section{Relevant Literature}\label{sec2:Lit}

\noindent Existing literature has addressed various models and case studies of the LMTS.  Several case studies analyze LMTS in different contexts, including \citet{liu2012solving}'s study of a bicycle-sharing program for a passenger LMTS in Beijing. Some studies have examined the design and performance evaluation of an LMTS from a planning perspective. For example, \citet{wang2016approximating} address the planning side by focusing on passenger LMTS from a stochastic and planning perspective and provide closed-form approximations for the performance of an LMTS as a function of the system's fundamental design parameters. \citet{zhu2020analysis} study passengers' multi-modal commuting behavior with ride-splitting and ride-sourcing systems, while considering their feeding effects on public transit---i.e., the ride-splitting fleet provides first- and last-mile services to public transit. 

Recent studies have examined the operation of an LMTS from an optimization perspective. For example, \citet{wang2019routing} focuses on LMTS from an operational perspective and provides efficient strategies for passenger assignment, vehicle routing, and scheduling operations based on a set of last-mile demand information. Similarly, \citet{agussurja2019state} study the use of ride-sharing in satisfying last-mile demands with the assumption that last-mile demands are uncertain and come in batches, and propose a two-level Markov decision process framework that is capable of generating a vehicle-dispatching policy. \citet{liu2019optimizing} focus on the fleet size and scheduling of feeder transit services while considering the influence of bike-sharing systems, propose several hybrid operation modes that combine fixed and dynamic frequencies in a bimodal period, and compare these with conventional bus scheduling with constant service frequencies. \citet{chen2020solving} focus on
solving the first-mile ride-sharing problem using autonomous vehicles and propose a mixed-integer linear programming model to determine autonomous vehicle dispatch and ride-sharing schemes for minimum operational costs. \citet{serra2019last} study the scheduling problem of last-mile service while considering uncertainty in the system, and propose a two-stage stochastic programming formulation for scheduling a set of known passengers and uncertain passengers that is realized
from a finite set of scenarios. \citet{chen2018pricing, chen2018fairness} study the pricing problem of multiple types of passengers in a LMTS using a queueing model to approximate passenger waiting time.

Personal rapid transit (PRT) and demand responsive transit (DRT), which
refer to a variety of on-demand transportation systems with characteristics that are similar, in some ways, to LMTS, have also attracted significant attention in recent years. For instance, the PRT system
control frameworks by \citet{anderson1998control}; financial
assessments by \citet{bly2005three} and \citet{berger2011personal}; performance approximations by \citet{lees2009ride}; and case studies by \citet{mueller2011simulation}.  Other papers focus on DRT concept
discussions, practical implementation, and assessment of simulations in case studies, such as \citet{horn2002multi} and \citet{quadrifoglio2008simulation}, among others. \blue{Relevant fleet sizing problems have also been studied for on-demand ride-pooling service by \citet{ke2020pricing}; one-way car sharing service by \citet{xu2019fleet}; and autonomous electric vehicles considering charging system planning by \citet{zhang2019joint}.}

\blue{In contrast to the increased awareness for LMTS modeling and methodology research, the availability of practical datasets is a limiting factor in LMTS research. In our numerical study, we construct a real case study based on New York City's travel demand data. In another paper, \citet{hao2021} curated last-mile transportation demand data arising from job-related commute in the United States, with an emphasis on the correlation between last-mile demand and household income level.}

\vspace{1mm}
For the challenge of stochastic and uncertain demand, we refer to some frameworks for optimization under uncertainty: stochastic programming, robust optimization, and distributionally robust optimization. When using stochastic programming, the goal is to optimize a certain measure of a random outcome (e.g., the expected operational cost) for a given fully known distribution of the uncertain parameters. We refer to \cite{birge2011introduction} and \cite{shapiro2014lectures} and references therein for thorough discussions about applications, formulations, and solution algorithms. Robust optimization (RO) assumes complete ignorance about the probability distribution of uncertain parameters. Instead, it assumes that an uncertain parameter's values may vary in a given constrained set, called ``\textit{uncertainty set}'' \citep{ben2015deriving,bertsimas2004price, soyster1973convex}. Optimization is performed with respect to the worst-case scenario in the uncertainty set, which may  inevitably lead to over-conservatism and suboptimal decisions for other more-likely scenarios \citep{chen2019robust, delage2010distributionally, thiele2010note}.  \blue{By focusing on the worst-case scenario, RO solutions are often overly conservative. Moreover, they usually have poor expected performances because they cannot capture the distributional information of uncertainty.}  \blue{Distributionally robust optimization (DRO) is a third approach to model uncertainty that bridge the gap between the conservatism of RO and the specificity of SP}. DRO optimal solutions are sought for the worst-case probability distribution within a family of candidate distributions, called an ``\textit{ambiguity set}''. One can use easy-to-approximate information such as the mean and range of random parameters to construct the ambiguity sets and models that better mimic reality and are less conservative than RO models. In addition, DRO models with some types of carefully designed ambiguity sets are often more tractable than their SP counterparts \citep{delage2010distributionally, rahimian2019distributionally}.


\section{Formulation and Analysis}\label{Sec:formulation}

\noindent  \textcolor{black}{In this section, we formally define the fleet sizing and allocation problem for the on-demand LMTS. We propose and analyze two optimization models, namely, a two-stage stochastic mixed-integer linear programming model in Section \ref{sec3.3:TSM} and a two-stage distributionally robust model in Section \ref{sec3.4:DR}.}  
\subsection{\textbf{Definitions and Random Parameters}}\label{sec3.1:statement}

\noindent \blue{We consider the perspective of a last-mile service provider who wants to determine the number of vehicles $m_s$ to allocate  to each service region $s \in S$ (e.g., an area around a metro station) in a particular city.}  Each service region $s \in S$ consists of a known number of last-mile stops $J_s$. For each stop $j \in J_s$ there is a known number, $K_s$, of feasible vehicle routes (i.e., a sequence of last-mile stops that a vehicle should visit on a trip). In each planing period (i.e., morning, afternoon, etc.), a set of punctual $I_s$ trains arrive at each service region $s \in S$.  Each train dispatches a batch of passengers demanding last-mile service, and allocated vehicles deliver them to their final last-mile destinations. The number of passengers (demand henceforth) that need rides to each last-mile stop $j \in J_s$ is random. The randomness of the demand stems from the uncertain number of passengers who make a last-minute request for last-mile service, in addition to those who request their last-mile service in advance.   We make the following assumptions as in \cite{wang2019routing}:

\begin{table}[t!]  
\small
\center
  \caption{Notation.}
\begin{tabular}{ll}
\hline
\multicolumn{2}{l}{\textbf{Indices}} \\
$i$& index of train \\
$j$ & index of last-mile stop \\
$k$ & index of route \\
\multicolumn{2}{l}{\textbf{Sets}} \\
$S$ &  set of service regions \\
$K_s$ & set of pre-selected feasible routes in service region $s$ \\
$I_s$ & set of trains arriving to bring passengers in sequence in service region $s$ \\
$J_s$ &  set of pre-specified last-mile stops in service region $s$ \\
\multicolumn{2}{l}{\textbf{Parameters}} \\
$M$ & maximum total number of vehicles in the fleet \\
$n_{i,j,s}$ & (random) number of passengers demanding last-mile stop $j$  arriving at the station from train $i$\\
$\nL_{i,j,s}/\nU_{i,j,s}$ & lower/upper bound of $n_{i,j,s}$  \\
$\phi_{j,k}$& 1 if last-mile stop $j$ is served by route $k$; 0 otherwise\\
$t_{k}$ & total travel time of route $k$, in terms of intervals between arrival trains\\ 
$t_{j,k}$&  travel time to last-mile stop $j$ on route $k$ \\
$c$ & vehicle capacity (i.e., number of seats) for vehicle\\
$h$ &  inter-arrival time (headway) between trains (demand batches)\\
$f_s$ & fixed cost of each vehicle \\
$\beta^{\mbox{\tiny w}}$& weight of passenger waiting time before boarding in the objective function\\
$\beta^{\mbox{\tiny r}}$ & weight of passenger in-vehicle riding time in the objective function\\
 \multicolumn{2}{l}{\textbf{Decision variables } } \\
 $m_s$ & number of vehicles allocated to service region $s$\\
$z_{i,j,k,s}$& number of passengers with destination at last-mile stop $j$ assigned to route $k$ right after arrival of train $i$ \\
$w_{i,k,s}$ & number of trips on route $k$ starting right after arrival of train  $i$ \\
 \multicolumn{2}{l}{\textbf{Intermediate variables } } \\
$u_{i,j,s}$ & number of unserved passengers with destination at last-mile stop $j$ waiting at the station \\
& after the arrival of train  $i$  \\
$v_{i,s}$ &number of available vehicles at the station after the arrival of train $i$ in service region $s$ \\
&and its corresponding service assignment\\
\hline
\end{tabular}\label{table:notation}
\end{table}


\begin{itemize}\itemsep0em
\item[A1.] The delivery fleet consists of at most $M$ vehicles, each with integer capacity $c$;
\item[\blue{A2.}] \blue{The set of last-mile stops in each service region is finite;}
\item[A3.] The set of feasible routes for LMTS vehicles in each service region is finite and preselected based on geometry, historical demand patterns, and some practical constraints---e.g., limits on the maximum number of last-mile stops on a single route or the route's maximum travel distance or travel time;
\item[A4.] The inter-arrival time (headway) between arrival trains is deterministic and equal to $h$. 
\end{itemize}

Given a fleet of at most $M$ vehicles and the sets of last-mile stops and preselected routes for LMTS vehicles, we aim to identify: (1) the number of vehicles to allocate for each service region, (2) a routing plan for the allocated vehicles  (i.e., route selection) in each service region, and (3) the assignment of passengers to vehicles for different realizations of demand patterns.  Decisions (1)--(2) are first-stage decisions that we make before observing demand patterns. If a route is selected, a vehicle should visit all of the last-mile stops specified on this route.  The assignment decisions (3) represent the recourse actions in response to the first-stage decisions and the realizations of demand patterns. The quality of fleet-allocation decisions is a function of the \textcolor{black}{vehicle fixed cost (which may include vehicle rental or purchase cost, etc.) plus a weighted sum of passenger waiting time before boarding a vehicle and in-vehicle riding time. }

\vspace{2mm}
\noindent \textbf{\textit{General notation}}: For $a,b \in \mathbb{Z}$, we define $[a]:=\{1,2,\ldots, a \}$ and $[a,b]_\mathbb{Z}:=\{ c \in \mathbb{Z}: a \leq c \leq b \}$. The abbreviations ``w.l.o.g.'' and ``w.l.o.o.'' respectively, represent ``without loss of generality'' and ``without loss of optimality.''  For notation brevity, we use ($I_s$, $J_s$, $K_s$) to denote both the number and set of (trains, last-mile stops, routes) in service region $s$.  For notational and modeling convenience, we assume w.l.o.g.\ that last-mile stops (and routes) are numbered sequentially, e.g., $J_1:=\{ 1,\ldots, 4 \}$, $J_2:=\{ 5, \ldots, 10\}$, $K_1:=\{ 1,2,3 \} $, $K_2:=\{ 4, 5, 6, 7 \}$, etc.

\subsection{\textbf{Two-stage Stochastic Model (\blue{SP}) for Fleet Sizing and Allocation}}\label{sec3.3:TSM}
\noindent In this section, we assume that we know the joint probability distribution $\mathbb{P}$ of the random number of passengers arriving at the station from train $i$ with a destination of last-mile stop $j$ in each service region $s$, $n_{i,j,s}$, for all $ s \in S, \ i \in I_s$ and $j \in J_s$ and formulate the problem as a two-stage a ``\textit{prior}'' stochastic mixed-integer programming model (\blue{SP}). In the first stage, we determine the number of vehicles to allocate to each service region, their routes, and the number of trips on each route. In the second stage, we assign passengers to vehicles for different realizations of demand and compute the associated riding and waiting-time costs for passengers. A priori optimization has a managerial advantage, since it guarantees the regularity of service, which is beneficial for both passengers and the service provider. 


Let integer decision variable $m_s$ represent the number of vehicles allocated to service region $s$. The feasible region $\mathcal{M}$ of variables $m$ is defined in \eqref{eq:RegionM} such that the number of allocated vehicles is less than or equal to the total number of vehicles in the fleet:
\begin{align}
\mathcal{M}&=\left\{ m:  \begin{array}{l} \sum_{s \in S}m_s \leq M. \end{array} \right\}  \label{eq:RegionM}
\end{align}

\blue{Let variable $w_{i,k,s}$ represent the number of trips on route $k$ starting right after the arrival of train $i$ at service region $s$, for all $i, k, s$. Let variable $v_{i,s}$ represent the number of vehicles waiting at the metro station after arrival of the train (demand batch) $i$ in service region $s$, for all $i, k, s$. Feasible region $\mathcal{W}$ of $(w,v)$ in \eqref{RegionW} defines and constrains the number of vehicles waiting at each metro station $s$ after the arrival of each train $i$ (demand batch) and its corresponding service assignment.}

\color{black}
\begin{align}\label{RegionW}
\mathcal{W}&=\left\{ (v,w):\begin{array}{l}  v_{0,s}= m_s- \sum \limits_{k \in K_s} w_{0,k,s}, \ \forall s \in S \\  
 v_{i,s}=v_{i-1,s}+ \sum \limits_{k \in K_s} w_{i-t_{k},k,s}-\sum \limits_{k \in K_s}w_{i,k,s}, \ \forall s \in S, \ \forall \blue{i \in I_s\setminus \{0\}}  \\
  v_{i,s} \geq 0, \ w_{i,k,s} \in \mathbb{Z}_,  \   \forall (i,j,k, s)   \end{array} \right\} 
\end{align}

\color{black}

Given a feasible $m \in \mathcal{M}$, $w \in \mathcal{W}$, and a joint realization of uncertain parameters $\xi:=(n_{i,j,s})$, we can compute: (1) the number of passengers $z_{i,j,k,s}$ with a destination of last-mile stop $j$ assigned to route $k$ right after the arrival of train $i$ at service region $s$, and (2) the number of unserved passengers $u_{i,j,s}$ with a destinations of each last-mile stop $j$ waiting at the metro station after the arrival of train (demand batch) $i$ at service region $s$ and its corresponding service assignment using the following linear program (see Table~\ref{table:notation} for notation):
\allowdisplaybreaks
\begin{subequations}\label{LP:metrics}
\begin{align}
  Q(m, w, \xi ):=  & \min \ \beta^{\mbox{\tiny w}} \sum \limits_{s\in S}\sum \limits_{i \in I_s} \sum \limits_{j \in J_s}  h u_{i,j,s}+ \beta^{\mbox{\tiny r}} \sum \limits_{s\in S} \sum \limits_{i \in I_s} \sum \limits_{j \in J_s} \sum \limits_{k \in K_s} t_{j,k} z_{i,j,k,s}
  \label{objSec:Metrics}\\
&  \ \  \text{s.t.} \   u_{0, j,s}= n_{0,j,s}- \sum \limits_{k \in K_s} z_{0,j,k,s} \phi_{j,k}, \forall s \in S, \ \forall  j \in J_s \label{Const1}\\
&  \blue{ u_{i,j,s}= u_{i-1,j,s}+ n_{i,j,s}- \sum \limits_{k \in K_s} z_{i,j,k,s} \phi_{j,k},  \forall s \in S, \  \forall i \in I_s\setminus \{0\}, \ \forall  j \in J_s \label{Const2}}\\
& \sum \limits_{j \in J_s} z_{i,j,k,s} \phi_{j,k} \leq c w_{i,k,s}, \ \forall s \in S, \ \forall  i \in I_s, \ \forall k \in K_s \label{Const5} \\
& \ u_{i,j,s}\geq 0,    \ z_{i,j,k,s} \geq 0, \ \forall (i,j,k, s)\label{Const6}
\end{align} \label{SecondStage}
\end{subequations}

The objective function \eqref{objSec:Metrics} minimizes a linear cost function of the total waiting time and riding time for passengers. Constraints \eqref{Const1} and \eqref{Const2} are passenger flow constraints---i.e., they define and constrain the number of unserved passengers with a destinations of each last-mile stop who are waiting at the metro station after the arrival of train $i_s$ and its corresponding service assignment. \blue{Constraint \eqref{Const5} ensures the vehicle service capacity restriction. Finally, constraint \eqref{Const6} specifies the feasible ranges of the decision variables}. Accordingly, we formulate the stochastic fleet sizing and allocation of fleet problem as follows:
\begin{align}
& \min \limits_{m \in \mathcal{M}, w \in \mathcal{W}} \Big (\sum_{s \in S} f_s m_s+\ \mathbb{E}_\mathbb{P} [Q (m, w, \xi)] \Big) \label{SP_model}
\end{align}

The \blue{SP} formulation in \eqref{SP_model} searches for vehicle sizing, allocation, routing, and scheduling decisions that minimize the vehicle fixed cost plus a weighted expected sum of passenger waiting time and riding time, \blue{where the expectation is taken with respect to a known joint probability distribution $\mathbb{P}$ of $\xi$, where $\xi$ is the a vector of demand (i.e., a vector of all $(n_{i,j,s}),$ for all $s \in S, i \in I_s$ and $j \in J_s$)}. The formulation generalizes the deterministic LMTS routing and scheduling formulation  of \cite{wang2019routing} by (1) considering multiple service regions, (2) considering fleet sizing and allocation decisions,  and (3) incorporating the uncertainty of passengers demand for LMTS. In Proposition~\ref{Prop:Binary}, \textcolor{black}{we show that the number of trips on routes that have at least two last-mile stops is at most one after each train arrival in the optimal solution}.



\begin{prop}\label{Prop:Binary}
For any $k \in K_s$, if $\sum_{j\in J_s}\phi_{j,k}\geq2$, then $w^{*}_{i,k,s}\leq 1$ in the optimal solution of formulation \eqref{SP_model}.
\end{prop}

\begin{proof} For any $k \in K_s$ such that $\sum_{j\in J_s}\phi_{j,k}=2$, w.l.o.g., assume that the last-mile stops served by route $k$ are visited in the sequence of $j_1, j_2$ in route $k$. We use $Q(w_{i,k,s})$ to denote the value of objective function \eqref{objSec:Metrics} for a solution with route $w_{i,k,s}=2$. The corresponding passenger assignments are $z_{i,j_1,k,s}$ and $z_{i,j_2,k,s}$ with $z_{i,j_1,k,s}+z_{i,j_2,k,s}\leq 2c$. Assume that route $k_1$ only serves stop $j_1$ and route $k_2$ only serves stop $j_2$.
\begin{itemize}
    \item If $z_{i,j_1,k,s}\leq c$ and $z_{i,j_2,k,s}\leq c$, we can construct another feasible solution with route  $w'_{i,k_1,s}=1$, $w'_{i,k_2,s}=1$, and passenger assignment $z'_{i,j_1,k_1,s}=z_{i,j_1,k,s}$ and $z'_{i,j_2,k_1,s}=z_{i,j_2,k,s}$. 
    \item If $c < z_{i,j_1,k,s} \leq 2c$ and $z_{i,j_2,k,s} < c$, we can construct another feasible solution with route  $w'_{i,k_1,s}=1$ and $w'_{i,k,s}=1$, and passenger assignment $z'_{i,j_1,k_1,s}=c$, $z'_{i,j_1,k,s}=z_{i,j_1,k,s}-c$ and $z'_{i,j_2,k,s}=z_{i,j_2,k,s}$. 
    \item If $z_{i,j_1,k,s} < c$ and $ c<z_{i,j_2,k,s} \leq 2c$, we can construct another feasible solution with route  $w'_{i,k_2,s}=1$ and $w'_{i,k,s}=1$, and passenger assignment $z'_{i,j_2,k_2,s}=c$, $z'_{i,j_2,k,s}=z_{i,j_2,k,s}-c$ and $z'_{i,j_1,k,s}=z_{i,j_1,k,s}$. 

\end{itemize}

Since routes $k_1$ and $k_2$ are sub-routes of route $k$, we have $t_{j_1,k_1}\leq t_{j_1,k}$, $t_{j_2,k_2}\leq t_{j_2,k}$, $t_{k_{1}}< t_{k}$ and $t_{k_{2}}< t_{k}$. Apparently, the value of objective function \eqref{objSec:Metrics} for the solution $Q(w'_{i,k,s}, w'_{i,k_1,s}, w'_{i,k_2,s}) < Q(w_{i,k,s})$, which means that the solution with $w_{i,k,s}=2$ is not the optimal solution. Similarly, we can justify that all solutions with $\sum_{j\in J_s}\phi_{j,k}>2$ and $w_{i,k,s}\geq 2$ are not the optimal solution. Therefore, in the optimal solutions, we have $w^{*}_{i,k,s}\leq 1$ for any $k\in K_s$ with $\sum_{j\in J_s}\phi_{j,k}\geq2$.

\end{proof}

\color{black}

Proposition \ref{Prop:Binary} implies that, at any vehicle dispatch decision after each train $i$ arrives, to reduce passenger riding time, we would never dispatch multiple identical routes (i.e., $w_{i,k,s}\geq2$) that visit multiple stops (i.e., $\sum_{j\in J_s}\phi_{j,k}\geq2$). In other words, if a route $k$ visits multiple stops (i.e., $\sum_{j\in J_s}\phi_{j,k}\geq2$), we would dispatch at most one vehicle to serve this route $k$ (i.e., $w_{i,k,s}\leq1$) after each train $i$ arrives.
This proposition will render many integer decision variables $w_{i,k,s}$ to binary decision variables, which can help to improve the computational efficiency of the optimization model.

\color{black}

\subsection{\textcolor{black}{\textbf{Distributionally Robust Model (DR) for Fleet Sizing and Allocation}}}\label{sec3.4:DR}

\noindent The \blue{SP} formulation in \eqref{SP_model} assumes that we know the probability distributions $\mathbb{P}$ of $\xi= (n_{i,j,s})$,	 where $ s \in S, \ i \in I_s$ and $j \in J_s$. However, in reality, it is challenging, if not impossible, to accurately identify (estimate) the true distribution of random parameters for the demand. In this section, we assume that $\mathbb{P}$ is not perfectly known. However, we know the support (i.e., upper and lower bounds) and the mean values of the random parameters. Mathematically, we consider support
\begin{align*}
&  \mathcal{R}:=\left\{ n \geq 0: \begin{array}{l} \underline{n}_{i,j} \leq n_{i,j,s} \leq \nU_{i,j,s},  \forall s \in S, i \in I_s, J \in J_s.  \end{array} \right\}
\end{align*}

In addition, we let $\mu:= \mathbb{E}[\xi]$ represent the mean (expected) value of $n$. Then we consider the following mean-support ambiguity set  $\mathcal{F}(\mathcal{R}, \mu)$:
\begin{align}\label{eq:ambiguity}
\mathcal{F}(\mathcal{R}, \mu) := \left\{ \mathbb{P} \in \mathcal{P}(\mathcal{R}): \begin{array}{l} \int_\mathcal{R} d\mathbb{P} = 1\\ \mathbb{E_P}[\xi] = \mu \end{array} \right\}
\end{align}
where $\mathcal{P}(\mathcal{R})$ in $\mathcal{F}(\mathcal{R}, \mu)$ represents the set of probability distributions supported on $\mathcal{R}$, and each distribution matches the mean values of $n$. Using the ambiguity set $\mathcal{F}(\mathcal{R}, \mu)$, we formulate the fleet sizing and allocation as the following min-max problem:
\begin{align}
& \min \limits_{m \in \mathcal{M}, w \in \mathcal{W}} \left \{ \sum_{s \in S} f_s m_s+ \sup \limits_{\mathbb{P} \in \mathcal{F}(\mathcal{R}, \mu)} \mathbb{E}_\mathbb{P} [Q(m, w, \xi)] \right \} \label{Obj:DRObjective2}
\end{align}

The formulation~\eqref{Obj:DRObjective2} seeks to identify vehicle sizing, allocation, routing, and scheduling decisions that minimize the worst-case expected cost of passenger waiting time and riding time over a family of distributions of random parameters residing in the ambiguity set $\mathcal{F}(\mathcal{R}, \mu)$ for the demand. $Q(m, w, \xi)$ is the recourse problem defined in \eqref{LP:metrics}.

\subsection{\textbf{Reformulation}}

\noindent In this section, we use duality theory and follow a standard approach in distributionally robust optimization to reformulate the min-max model in \eqref{Obj:DRObjective2} to one that is solvable.  We first consider the inner maximization problem $\sup \limits_{\mathbb{P} \in \mathcal{F}(S,  \mu)}  \mathbb{E}_\mathbb{P} [Q (m, w, \xi)]$ for a fixed vehicle allocation decision $m \in \mathcal{M}$ and \textcolor{black}{$w \in \mathcal{W}$}, where $\mathbb{P}$ is the decision variable---i.e., we are choosing the distribution that maximizes the expected value of $Q (m, w, \xi)$.
\begin{subequations}\label{Problem:InnerMax}
\begin{align}
& \max   \ \mathbb{E_P}[Q(m, w, \xi)]  \\
& \ \text{s.t.}  \ \ \mathbb{E_P}[\xi] = \mu, \label{Problem:InnerMax-C1}  \\
& \ \ \ \ \  \ \ \mathbb{E_P}[\mathds{1}_\mathcal{R}(\xi)] = 1 \label{Problem:InnerMax-C2}
\end{align} 
\end{subequations}
where $\mathds{1}_\mathcal{R}(\xi)= 1$ if $\xi \in \mathcal{R}$ and $\mathds{1}_\mathcal{R}(\xi) = 0$ if $\xi \notin \mathcal{R}$.  As we show in the proof of Proposition~\ref{Prop:DualMinMax}, problem \eqref{Problem:InnerMax} is equivalent to problem \eqref{eq:FinalDualInnerMax-1}.

\color{black}
\begin{prop}\label{Prop:DualMinMax}
For a fixed \blue{$m \in \mathcal{M}$} and \blue{$ w \in \calW$}, problem \eqref{Problem:InnerMax} is equivalent to 
\begin{align}  \label{eq:FinalDualInnerMax-1}
& \min  \ \ \Big [ \sum \limits_{s \in S} \sum \limits_{i \in I_s} \sum \limits_{j \in J_s} \mu_{i,j,s}\rho_{i,j,s}  + \max \limits_{n \in \mathcal{R}} \{ \textcolor{black}{Q(m,w,\xi)} +  \sum \limits_{s \in S} \sum \limits_{i \in I_s} \sum \limits_{j \in J_s} - n_{i,j,s} \rho_{i,j,s} \} \Big] 
\end{align}
\end{prop}

\color{black}

\begin{proof} For a fixed $(m,w)$  we can formulate problem \eqref{Problem:InnerMax} as the following linear functional optimization problem: 
\begin{subequations}\label{InnerMax}
\begin{align}
& \max_{\mathbb{P} \geq 0} \ \int_{\mathcal{R}}Q(m, w, \xi)  \ d \mathbb{P}  \\
& \ \text{s.t.} \ \   \int_{\mathcal{R}} \blue{n_{i,j,s}} \ d\mathbb{P}= \mu_{i,j,s}
  \quad \quad \forall s \in S, \ \forall i \in I_s, \ \forall j \in J_s \label{ConInner1}\\
& \ \ \ \ \  \ \ \int_{\mathcal{R}}  d\mathbb{P}= 1 \label{ConInner:Distribution}
\end{align} 
\end{subequations}

\color{black}

\noindent Letting  $\rho$ and $\theta$ be the dual variable associated with constraints \eqref{ConInner1} and \eqref{ConInner:Distribution}, respectively, we present problem \eqref{InnerMax} in its dual form:
\begin{subequations}
\begin{align}
& \min_{(\rho, \alpha, \lambda, \theta)}  \ \ \ \ \ \ \ \ \ \ \sum \limits_{s \in S} \sum \limits_{i \in I_s} \sum \limits_{j \in J_s} \mu_{i,j,s} \rho_{i,j,s} + \theta \label{DualInner:Obj} \\
& \ \qquad \ \  \text{s.t.} \ \ \ \ \ \ \  \sum \limits_{s \in S} \sum \limits_{i \in I_s} \sum \limits_{j \in J_s}  n_{i,j,s} \rho_{i,j,s}+  \theta \geq \textcolor{black}{Q(m,w,\xi)},  && \forall n \in \mathcal{R} \label{DualInner:PrimalVariabl}
\end{align} \label{DualInnerMax}
\end{subequations}
\color{black}
\noindent where $\rho$ and $\theta$ are unrestricted in sign, and constraint \eqref{DualInner:PrimalVariabl} is associated with the primal variable $\mathbb{P}$. Under the standard assumption that $\mu$ belongs to the interior of the set $\{ \int_{\mathcal{R}} n_{i,j,s}\mathbb{Q}: \mathbb{Q}$ is a probability distribution over support $\mathcal{R} \}$, strong duality holds between \eqref{InnerMax} and \eqref{DualInnerMax} (see \cite{bertsimas2005optimal} for a detailed discussion of this assumption and \cite{jian2017integer, shehadeh2020distributionally,shehadeh2021Sanci} for applications). \blue{Note that for fixed ($\rho, \theta$), constraints \eqref{DualInner:PrimalVariabl} are  equivalent to $\theta \geq \max \limits_{ n \in \mathcal{R}} \{ \textcolor{black}{Q(m,w,\xi)} - \sum \limits_{s \in S} \sum \limits_{i \in I_s} \sum \limits_{j \in J_s}  n_{i,j,s} \rho_{i,j,s} \} $.  Since we are minimizing $\theta$ in \eqref{DualInnerMax}, the dual formulation of \eqref{InnerMax} is equivalent to}
\color{black}
\begin{align*} 
& \min \ \ \Big [ \sum \limits_{s \in S} \sum \limits_{i \in I_s} \sum \limits_{j \in J_s} \mu_{i,j,s}\rho_{i,j,s}  + \max \limits_{n \in \mathcal{R}} \{ \textcolor{black}{Q(m,w,\xi)} +  \sum \limits_{s \in S} \sum \limits_{i \in I_s} \sum \limits_{j \in J_s} - n_{i,j,s} \rho_{i,j,s} \} \Big] 
\end{align*}
\end{proof}

\color{black}

\blue{Note that the recourse problem $Q(m, w, \xi)$ is a minimization problem}. Thus, in  \eqref{eq:FinalDualInnerMax-1}  we have an inner max-min problem. Next, we use $Q(m,w, \xi)$ properties to derive an equivalent single minimization reformulation of \eqref{eq:FinalDualInnerMax-1}. For fixed $m \in \mathcal{M}$, $w \in \mathcal{W}$, and a realized value of $\xi=n$, $ Q(m, w, \xi)$ is a linear program. We formulate $ Q(m, w, \xi)$ in its dual form as 
\begin{subequations}\label{DualofQ}
\begin{align}
 Q(m, w, \xi)=& \max_{\Gamma, \psi} \Big \{ \sum \limits_{s \in S} \sum \limits_{i \in I_s} \sum \limits_{j \in J_s}  n_{i,j,s} \Gamma_{i,j,s}+ \sum \limits_{s \in S} \sum \limits_{i \in I_s} \sum \limits_{k \in K_s} c w_{i,k,s} \psi_{i,k,s}  \Big \}&\\
 & \st \ \  \Gamma_{i,j,s}- \Gamma_{i+1,j,s} \leq \beta^{\mbox{\tiny w}} h, \qquad \forall (s, i, j) \label{ConstDual1}\\
 & \qquad \phi_{j,k} (\Gamma_{i,j,s}+ \psi_{i,k,s}) \leq t_{j,k} \beta^{\mbox{\tiny r}}, \qquad \forall (i, j, k, s) \label{ConstDual2}\\
 & \qquad \psi_{i,k,s} \leq 0, \ \Gamma_{I+1,j,s}=0, \qquad \forall (i, j, k, s)\label{ConstDual3}
\end{align}
\end{subequations}
where $\Gamma$ and $\psi$ are the dual variables associated with constraints \eqref{Const1}--\eqref{Const2} and \eqref{Const5}, respectively.  \blue{Given the dual formulation of $ Q(m, w, \xi)$ in \eqref{DualofQ} and a fixed and feasible ($\pmb{\rho, w}$), we can rewrite the inner maximization problem $\max \limits_{n \in \calR} \{ Q(m,w,\xi)  +  \sum \limits_{s \in S} \sum \limits_{i \in I_s} \sum \limits_{j \in J_s} - n_{i,j,s} \rho_{i,j,s}\}$ in \eqref{eq:FinalDualInnerMax-1} as follows}
\begin{subequations}\label{Inner2}
\begin{align}
\max \limits_{\Gamma, \psi, \nb \in [\pmb{\nL, \nU}]} & \ \  \Big \{ \sum \limits_{s \in S} \sum \limits_{i \in I_s} \sum \limits_{j \in J_s}  n_{i,j,s} (\Gamma_{i,j,s}-\rho_{i,j,s})+ \sum \limits_{s \in S} \sum \limits_{i \in I_s} \sum \limits_{k \in K_s} c w_{i,k,s} \psi_{i,k,s}    \Big\} \\
\ \ \ \st \ & \eqref{ConstDual1}-\eqref{ConstDual3}
\end{align}
\end{subequations}
As we show in the proof of Proposition \ref{Prop2} in \ref{Appex:Proof_Prop2}, for fixed $(m,  w, \xi)$ and $\rho$, problem \eqref{Inner2} is equivalent  to the minimization problem in \eqref{Inner2_Final_main} .
\begin{prop}\label{Prop2}
\blue{For fixed $(m,  w, \xi)$ and $\rho$ problem \eqref{Inner2}  is equivalent to}
\begin{subequations}\label{Inner2_Final_main}
\begin{align}
\min \limits_{y \geq 0,x\geq 0} & \  \sum \limits_{s \in S} \sum \limits_{i \in I_s} \sum \limits_{j \in J_s} \Big[ \beta^{\mbox{\tiny w}} h y_{i,j,s}+ \sum \limits_{k \in K_s} t_{j,k} \beta^{\mbox{\tiny r}} x_{i,j,k,s} \Big]-\sum \limits_{s \in S} \sum \limits_{i \in I_s} \sum \limits_{j \in J_s} \nU_{i,j,s} \rho_{i,j,s} \\
 \st & \ y_{0,j,s}+\sum \limits_{k \in K_s} \phi_{j,k} x_{0,j,k,s} \geq  \nU_{0,j,s}, \qquad \forall s \in S,  \ j \in J_s  \\
&  \ y_{i,j,s} -y_{i-1,j,s}+ \sum \limits_{k \in K_s} \phi_{j,k} x_{i,j,k,s} \geq   \nU_{i,j,s}, \qquad \forall s \in S, \  i \in I_s\setminus \{0\}, \  j \in J_s \\
&  \  \sum \limits_{j \in J_s} \phi_{j,k} x_{i,j,k,s}  \leq cw_{i,k,s}, \qquad \forall s \in S, \ i \in I_s, \ j \in J_s\\
& \ m \in \mathcal{M}, \ w \in \mathcal{W}, \  y \geq 0, \ x\geq 0
\end{align}
\end{subequations}

\end{prop}
\blue{Combining the inner problem in the form of \eqref{Inner2_Final_main} with the outer minimization problems in  \eqref{eq:FinalDualInnerMax-1} and \eqref{Obj:DRObjective2}, we derive the following equivalent MILP reformulation of the DR model in \eqref{Obj:DRObjective2} (see \ref{Appex:Proof_Prop2})}
\begin{subequations}\label{DR_Final_reform}
\begin{align}
 \min  &\Bigg \{ \sum_{s \in S} f_s m_s+ \sum \limits_{s \in S} \sum \limits_{i \in I_s} \sum \limits_{j \in J_s}( \mu_{i,j,s}- \nU_{i,j,s} )\rho_{i,j,s}+   \sum \limits_{s \in S} \sum \limits_{i \in I_s} \sum \limits_{j \in J_s} \Big[ \beta^{\mbox{\tiny w}} h y_{i,j,s}+ \sum \limits_{k \in K_s} t_{j,k} \beta^{\mbox{\tiny r}} x_{i,j,k,s} \Big] \Bigg \}\\
 \st &  \ m \in \mathcal{M}, \ w \in \mathcal{W}, \  y \geq 0, \ x\geq 0 \\
& \ y_{0,j,s}+\sum \limits_{k \in K_s} \phi_{j,k} x_{0,j,k,s} \geq  \nU_{0,j,s}, \qquad \forall s \in S,  \ j \in J_s  \\
&  \ y_{i,j,s} -y_{i-1,j,s}+ \sum \limits_{k \in K_s} \phi_{j,k} x_{i,j,k,s} \geq   \nU_{i,j,s}, \qquad \forall s \in S, \  i \in I_s\setminus \{0\}, \  j \in J_s \\
&  \  \sum \limits_{j \in J_s} \phi_{j,k} x_{i,j,k,s}  \leq cw_{i,k,s}, \qquad \forall s \in S, \ i \in I_s, \ j \in J_s
\end{align}
\end{subequations}

\section{\textcolor{black}{Computational Experiments and Implications}}\label{sec:CompResults}
\noindent  The primary objective of our computational study is to evaluate the computational and operational performance of the proposed models. We solve the two-stage \blue{SP} in \eqref{SP_model} via the the sample average approximation (SAA) approach in \ref{Appx:SAA} (see, e.g., \cite{kim2015guide,kleywegt2002sample} for a detailed discussion of SAA).  Section~\ref{sec5.1:Expsetup} presents the details of data generation and experimental design. In Section \ref{sec5.2:CPU}, we evaluate the solution times of the \blue{SP} and DR models. In Section~\ref{sec5;optimalsolutions}, we evaluate the optimal solutions of the \blue{SP} and DR models and their out-of-sample simulation performance. We close by analyzing the sensitivity of the optimal solutions to different parameter settings in Section~\ref{sec5.4:sensitivity}.

\subsection{\textcolor{black}{\textbf{Experimental Design and Computational Setup}}}\label{sec5.1:Expsetup}

\noindent  We first construct four instances (instance 1-4 henceforth), in part based on the parameter settings and assumptions made by \cite{wang2019routing} (which address the deterministic counterpart LMTS routing and scheduling problem for one service region). We summarize our test instances in Table~\ref{table:DRAFTInstances}. Each of the four instances is characterized by the number of regions $S$, number of last-mile stops in each region $J_s$, and number of routes in each region $K_s$. The sizes of the instances vary and correspond to different practical contexts. \blue{In addition to these four instances, we then construct an instance based on the actual on-demand transportation data related to New York City (NYC instance henceforth)\footnote{\blue{Souce: https://www1.nyc.gov/site/tlc/about/tlc-trip-record-data.page. The dataset contains the yellow and green taxi trip records include fields capturing pick-up and drop-off dates/times, pick-up and drop-off locations, and driver-reported passenger counts. It was collected and provided to the NYC Taxi and Limousine Commission (TLC) by technology providers authorized under the Taxicab \& Livery Passenger Enhancement Programs (TPEP/LPEP). Although it is not a real demand record for an exact existing LMTS, since the dataset contains real information of on-demand passengers, which also reflects the actual spatial and temporal patterns and uncertainty of the demand.}}, which consists of $S=4$ regions. The procedure to construct the NYC instance, the details of last-mile stops and routes, and the empirical statistics of batch demand are summarized in \ref{AppexNYC}.}

To generate demand profile for the instance 1-4, we use a similar magnitude of the number of passengers in the LMTS literature (e.g., \cite{wang2016approximating,wang2019routing}), as well as the same random parameter generation procedures in the distributionally robust scheduling and optimization literature (e.g., \cite{jian2017integer, shehadehDMFRS, mak2014appointment}). For instance 1--4,  we randomly generate the mean values $\mu$ of $n$ from a uniform distribution ${U[ 1, 4]}$ and standard deviation $\sigma=0.5\mu$. We randomly generate $N$ in-sample realizations $(n_{i,j,s}^{\mbox{\tiny n}})$ by following lognormal (LogN) distributions with the generated   $\mu_{i,j,s}$  and $\sigma_{i,j}$,  for all $ s\in S, i \in I_s, j \in J_s, n \in N$. LogN is a standard distribution of model customers' demand and service times in a wide range of applications. \cite{kamath2002bayesian} results suggest that the LogN is a suitable distribution for modeling stochastic demands in an economic context. \blue{For the NYC instance, in Table~\ref{table:NYCStats} in \ref{AppexNYC}, we present the empirical $\mu$ and $\sigma$ of batch demand $n_{i,j,s}$.}

According to \cite{gomez1999essays}, for work trips in San Francisco, the monetary value of a unit of transfer waiting time is 195$\%$ of the user's after-tax wages, and the monetary value of a unit of in-vehicle riding time is 76$\%$ of the user's after-tax wages. In general, we should have $\beta^{\mbox{\tiny w}}>\beta^{\mbox{\tiny r}}$ in the objective function. Following similar parameter selections as in \cite{wang2019routing}, we normalize $\beta^{\mbox{\tiny r}}=1$ and $\beta^{\mbox{\tiny w}}\in[2,3]$ (we perform sensitivity analysis in Section~\ref{sec5.4:sensitivity}). As for the vehicle fixed cost $f_s$, if the service provider already has a fleet of vehicles, $f_s$ could be very small; if the service provider rents vehicles, $f_s$ should include the rental fee and operating cost (e.g., fuel and gas); if the service provider needs to purchase a new fleet of vehicles, $f_s$ may be very large and a complex depreciation 
should be considered. \textcolor{black}{Unless stated otherwise, we test two values of $f_s$: (1) $f_s=0$ (ignoring the fixed cost and assuming an existing vehicle fleet), and (2) $f_s=4,000,$ from a range of [2,000, 6,000] (renting vehicles to serve passengers who are less sensitive to riding and waiting times)\footnote{The range is approximated considering the general after-tax wage for city residents, the rental price of vehicles with capacity 4-10, vehicle price and possible depreciation, fuel cost, etc.} (considering vehicle rental fee and/or depreciation and operating cost; see Appendix \ref{Appix:fixed cost}). We investigate the impact of $f_s$ in Section~\ref{sec5.4:sensitivity}. }

\vspace{1mm}
\textcolor{black}{As in prior applied distributionally robust literature (see, e.g., \cite{jian2017integer}, \cite{shehadeh2021Sanci}, \cite{shehadeh2020distributionally}, and references therein), we respectively use the $20\%$-quantile and $80\%$-quantile values of the $N$ in-sample data to approximate the lower $\underline{n}$ and upper $\overline{n}$ bounds of $n$. We optimize the SP model by using all of the $N$ data points, and the DR model with the corresponding mean,  lower bounds, and upper bounds.} We implemented the models using AMPL2016 programming language calling CPLEX V12.6.2 as a solver with default settings. We ran all experiments on a computer with an Intel Core i7 processor, 2.5 GHz CPU, and 16 GB (1600MHz DDR3) of memory, and imposed a solver time limit of 1 hour.



\begin{table}[t!]
 \center 
 \footnotesize
   \renewcommand{\arraystretch}{0.5}
  \caption{Four Instances. Notation: $S$ is \# of regions, $s$ is a region, $I_s$ is \# of trains, $J_s$ is \# of last-mile stops in regions $s$, $K_s$ is number of routes in region $s$.}
\begin{tabular}{ccccccccccccc}
\hline
\textbf{Inst} & \textbf{$S$} & $s$ &  \textbf{$I_s$}&  \textbf{$J_s$}   & \textbf{$K_s$} &&  \textbf{Inst} & \textbf{$S$} & $s$   &  \textbf{$I_s$} &  \textbf{$J_s$}   & \textbf{$K_s$} \\
\hline
1 &  4   & $ \ \ s=1$ & 12 & 4 & 10 & & 2&  4 & $ \ \ s=1$ & 12 & 4 & 13 \\
&  & $ \ \ s=2$ & 12 & 6& 23 &&& &$ \ \ s=2$ & 12 & 6 & 31 \\
&& $ \ \ s=3$ & 12 & 6& 30 && &&$ \ \ s=3$ & 12 & 6 & 24 \\
&& $ \ \ s=4$ & 12 & 8& 39 && & & $ \ \ s=4$ & 12 & 8 & 40 \\
\\
 3 & 5  & $ \ \ s=1$ & 12 & 4 & 13&  &4 & 6 &  $ \ \ s=1$ & 12 & 4 & 13\\
& &$ \ \ s=2$ & 12 & 6 & 31    &&&  & $ \ \ s=2$ & 12 & 6 & 31 \\
&& $ \ \ s=3$ & 12 & 6 & 24    &&& &  $ \ \ s=3$ & 12 & 6 & 24 \\
&& $ \ \ s=4$ & 12 & 8 & 40   && & & $ \ \ s=4$ & 12 & 8 & 40 \\
& &$ \ \ s=5$ & 12 & 8 & 49   && & & $ \ \ s=5$ & 12 & 8 & 49 \\
 && & & & & &&  &$ \ \ s=6$ & 12 & 8 & 59 \\
\hline																									
\end{tabular} 
\label{table:DRAFTInstances}
\end{table}


\subsection{\textbf{CPU Time}}\label{sec5.2:CPU}

\noindent \textcolor{black}{In this section, we analyze the solution times of the \blue{SP} and DR models. For each instance in Table \ref{table:DRAFTInstances}, we first generate mean demand $\mu$ for each last-mile stop in each service region from $U[1,4]$ (low to average demand) and $U[3, 7]$ (high demand), and set $\sigma=0.5 \mu$. Then we generate the in-sample data from LogN using the generated $\mu$ and $\sigma$. To approximate the lower $\underline{n}$ and upper $\overline{n}$ of $n$, we respectively use the $20\%$-quantile and $80\%$-quantile values of the $N=100$ in-sample data.  We optimize the \blue{SP} by using all of the $N$ data points, and the DR model with the corresponding corresponding mean,  lower bounds, and upper bounds.  We use $N=100$, $f_s \in \{ 0, 4,000 \}$, and $M \in \{ 40, 60, 80\} $. For each instance, we impose a time limit of 7,200 seconds (i.e., 2 hours).}

\vspace{1mm}
 \textcolor{black}{Our choice of the sample size $N$ to solve the SAA of \blue{SP} was motivated by the trade-off between the computational effort required to solve the resulting mixed-integer linear programs (MILPs) and the quality of approximation of the expected value objective of \blue{SP} by its SAA approximation.  On the one hand, the sizes of MILP instances increase with $ N $ and their solution times also increase. On the other hand, optimal solutions of SAA instances with larger values of $ N $ are likely to be closer to optimality compared with the expected value objective.}
\vspace{1mm}

\textcolor{black}{The literature on the SAA method provides theoretical insights and guidance for selecting a sample size from this perspective. We implemented the so-called Monte Carlo Optimization (MCO) procedure to compute statistical lower and upper bounds on the optimal value of \blue{SP} based on an optimal solution to its SAA approximation, which in turn provides a statistical estimate of the relative approximation gap between the optimal value of \blue{SP} and its SAA approximation (see, e.g., \cite{homem2014monte} and \cite{linderoth2006empirical} for a thorough discussion of MCO). Applying the MCO procedure to our \blue{SP} model with $N = 100$, we estimate the relative approximation gaps for the \blue{SP} instances described in Table~\ref{table:DRAFTInstances} to range between 1\% and 5\%. In contrast, larger sample sizes resulted in longer solution times without consistent and significant improvements in the relative approximation gaps. Based on these considerations, we selected $N = 100$ for our computational experiments.}


\begin{table}[t!]
\center 
 \footnotesize
   \renewcommand{\arraystretch}{0.9}
\caption{Solution times (in seconds) using the DR and SAA of \blue{SP}. ``--'' indicates termination without any feasible MIP solutions (and thus no upper bound).} 
\begin{tabular}{cclllllllllllllllllllllll}
 \hline
\textbf{Inst}&  \textbf{Range} & \textbf{$\pmb{M}$} & $\pmb{f_s}$ &  \textbf{DR} & \textbf{SAA}  & \textbf{Range} &  \textbf{$\pmb{M}$} & $\pmb{f_s}$ &  \textbf{DR} & \textbf{SAA}\\
 \hline
1 &   $[1,4]$ & 40 & 0 & 0.30 & 197 &  $[3,7]$ & 40 & 0 & 1.67& 2,477 \\
& & &   4,000 & 1.7 & 7,110 &  &  & 4,000 & 0.2 & 7,200\\ 
\\
	& & 60  & 0 & 0.25& 60 &  & 60 & 0 &2  & 305 \\
	& & &   4,000 & 25& 7,200 & & & 4,000 & 13 & 7,200\\

\\
	& & 80   & 0 & 0.22& 31& & 80 &0 & 0.27 & 230 \\
 	& & &  4,000 & 18 & 7,200& & & 4,000 &12 & 2,542 \\

\\

2 & $[1,4]$  & 40 & 0& 0.3 & 7,200 &   $[3,7]$ & 40 & 0 & 1 &7,200  \\
				& & & 4,000& 2  & 7,200&  & &  4,000 &  1  &7,200 \\
				\\
		& & 60 & 0	& 1 & 0.13 	& &   60 & 0 & 1&  91 \\
		& & & 4,000&  2 &	7,200& && 4,000& 1 & 7,200\\	
		\\
		& & 80 & 		0 &  0.16& 20 & & 80 &   0 & 0.3&    54 \\
		& & & 4,000&   2& 7,200 &&&  4,000  & 2 & 4,290\\
		
\\
3 &  $[1,4]$ & 40 & 0 & 9 & 7,202 &  $[3,7]$ & 40 & 0 & 1 & 7,200  \\
 & & & 4,000& 3 & 7,200&  & & 4,000&  1 & 7,200\\
 \\
 
	& & 60 & 0	&	&500 &  & 60 & 0 & 9&	7,200 \\
	& & & 4,000& 38 & 7,200 & && 4,000 & 2 & 7,200\\

	\\
	& & 80 & 0&	1& 38& & 80 &0 &  3 & 160	\\
	& & & 4,000& 2.3 & 7,200  & & & 4,000&  2&  7,200\\
	\\
4 &  $[1,4]$ & 40 & 0 & 3,600& \multicolumn{1}{c}{--} &  $[3,7]$ & 40 & 0 & 1& \multicolumn{1}{c}{--}\\
 & & & 4,000& 32 & \multicolumn{1}{c}{--}&  & & 4,000&1  & \multicolumn{1}{c}{--}\\	
 \\
 & & 60 & 0	&4	&\multicolumn{1}{c}{--} &  & 60 & 0 &1 & \multicolumn{1}{c}{--}\\
	& & & 4,000&22 &\multicolumn{1}{c}{--} & && 4,000 & 1& \multicolumn{1}{c}{--}\\
	\\
	 & & 80 & 0	&	0.30& \multicolumn{1}{c}{--}&  & 80 & 0 & 4& 10\% \\
	& & & 4,000&50 & \multicolumn{1}{c}{--}& && 4,000 & 40 & 33\% \\
\hline
\end{tabular}\label{table:SolutionCPU}
\end{table}

\vspace{1mm}

Table~\ref{table:SolutionCPU} presents solution times (in seconds) using the \blue{SP} and DR models for different values of $M$ (maximum number of vehicles). We first observe that the DR can quickly solve all instances under the two ranges of the average number of passengers and all values of $f_s$ and $M$ and significantly faster than the \blue{SP} model. In contrast, the \blue{SP} fails to solve all of the SAA instances corresponding to instance 4 to optimality within the time limit, and terminates with either a large relative MIP (relMIP) gap (relMip:=$\frac{UB-LB}{LB}$, where UB is the best upper bound and LB is the linear programming relaxation-based lower bound obtained at termination) or without any feasible MIP solution, and thus no upper bound. Additionally, the \blue{SP}'s solution times differ under the two ranges of the average number of passengers and values of $f_s$ and $ M $.

\textcolor{black}{ Under $f_s=0$, the two models allocate all vehicles (i.e., $\sum_{s \in S}m_s^*=M)$.  The \blue{SP}'s solution times decrease as $M$ increases from 40 to 80. Consider instance 1, for example: \blue{SP}'s solution times decrease from 197 sec and 2,477 sec to 31 sec and 230 sec, respectively, under uncertain demand [1,4] and [3,7].  A possible explanation for this is that when $ M $ is large, we can satisfy passenger demand with any vehicle allocations, routing, and scheduling decisions. That is, we can allocate a larger number of vehicles in each region that may be sufficient to transport passengers to their last-mile stops via the direct routes to those steps (e.g., each route serves only one stop). When $M$ is small, vehicle allocations, routing, and scheduling decisions are subtle and difficult to optimize. }

\textcolor{black}{Under $f_s=4,000$, the two models allocate a subset of the $M$ vehicles to minimize the total cost function (see Table \ref{table:OptAllocAndRoutes}). The larger \blue{SP} solution times indicate that \blue{SP}'s routing and scheduling decisions are subtle and difficult to optimize in this case.}

\subsection{\textbf{Analysis of optimal solutions}}\label{sec5;optimalsolutions}

\noindent \textcolor{black}{In this section, we compare the DR and \blue{SP} optimal vehicle sizing and allocation decisions and their out-of-sample performance using the same settings as in Section~\ref{sec5.2:CPU}. Under zero vehicle fixed cost (i.e., $f_s=0$), the two models have similar performance for all three instances. \textcolor{black}{We show such results for instance 1 in Table~\ref{table:InSamplef=0} in \ref{Appex:InSampleZeroF}}. Therefore, in this section we mainly compare the optimal sizing and allocation decisions under $f_s=4,000$. For presentation brevity and illustrative purposes, we fix $M=60$ and present results for instances 1--3 \blue{and NYC}, as the SAA-\blue{SP} can solve all such instances. }

\vspace{1mm}

\textcolor{black}{First, we analyze optimal vehicle sizing and allocation decisions yielded by the DR and \blue{SP} models, which are presented in Table~\ref{table:OptAllocAndRoutes}.  From this table, we first observe that, \blue{for instances 1-3}, both models allocate a higher number of vehicles to each service region under demand [3,7] than [1,4]. This makes sense, as a larger batch of passengers arrives in each train in the former case. Second, we observe that by incorporating ambiguity in the number of passengers arriving at each station after each train in each service region, the DR models always allocate more vehicles than the \blue{SP}.   As we show next, allocating more vehicles results in a better quality of service in terms of passenger waiting and riding  times, but a higher vehicle fixed cost.}

\begin{table}[t!]

\center 
 \footnotesize
   \renewcommand{\arraystretch}{0.9}
\caption{Optimal sizing and allocation decisions with $f_s=4,000$ and $M=60$} 
\begin{tabular}{cclclllllllllllllllllllll}
\hline
\textbf{Inst}&\textbf{Range}& \textbf{Model} & $\pmb{\sum_{s \in S} m_s^*}$ & $\pmb{m_1^*}$ &$\pmb{m_2^*}$ & $\pmb{m_3^*}$ &  $\pmb{m_4^*}$\\
\hline
1 & [1,4]& DR & 28 & 4& 7 & 7& 10\\
 &  & \blue{SP}  & 16 & 2 & 4 & 4 & 6\\
 \\
 & [3,7] & DR & 59& 11 & 16 & 12 & 20 \\ 
 &        &  \blue{SP}& 36& 7& 9 & 8 & 12 \\
 \\
 2 & [1,4] & DR & 22 & 2& 6 & 6& 8\\
 &  & \blue{SP} & 17 & 2 & 4 & 5 & 6 \\
  \\
  & [3,7] & DR & 50 & 7 & 11 & 12 & 19 \\
  &         & \blue{SP}& 28 & 5& 7 & 7 & 10\\
  \\
  \blue{NYC} &  & \blue{DR} & \blue{21} &  \blue{3}& \blue{6}& \blue{7}& \blue{5} \\
            &   & \blue{SP} &  \blue{10} &  \blue{2}& \blue{3}& \blue{3}& \blue{2}\\
            \\
  \hline
\textbf{Inst}&\textbf{Range}& \textbf{Model} &   $\pmb{\sum_{s \in S} m_s^*}$& $\pmb{m_1^*}$ &$\pmb{m_2^*}$ & $\pmb{m_3^*}$ &  $\pmb{m_4^*}$ & $\pmb{m_5^*}$\\
\hline 
 3 & [1, 4]& DR & 31 & 2 & 6 & 6 &8& 9 \\ 
 	&			&\blue{SP}& 20 & 2&4&4&5&5\\
 	\\
 	& [3, 7] & DR & 60 & 8& 10 & 11& 16 & 15 \\
 	&			& \blue{SP}& 42& 5& 7& 7& 11& 12\\
\hline
\end{tabular}\label{table:OptAllocAndRoutes}
\end{table}

\vspace{1mm}

Next, we analyze the in-sample performance of the optimal vehicle sizing and allocation decisions of DR and \blue{SP} under ``perfect information'' (known distributions) and out-of-sample performance with ``misspecified distribution information.'' 
Specifically, we simulate the optimal solutions of DR and \blue{SP} using the following two sets of $N^\prime=10,000$ out-of-sample data $n_{i,j,s}^{\mbox{\tiny 1}}, \ldots, n_{i,j,s}^{\mbox{\tiny N}^\prime} $ , for all $s \in S, i \in I_s, j \in J_s$.
\color{black}
\begin{enumerate}\itemsep0em
\item \textit{Perfect information}: We use the same parameter settings in Sections~\ref{sec5.1:Expsetup}-\ref{sec5.2:CPU} that we use to generate the $N$ in-sample data to generate the $N^\prime$ data from LogN. This is to simulate the in-sample performance.
\item \textit{Misspecified distribution information}: To simulate the out-of-sample performance of the DR and \blue{SP} optimal solutions when the in-sample data are biased, we keep the same mean $\mu_{i,j,s}$, standard deviation $\sigma_{i,j}$, and range values of $n$ as in the in-sample data, but we vary the distribution type of $n_{i,j,s}$ to generate the $N^\prime$ data. Specifically, we follow a Uniform distribution to generate realizations $n_{i,j,s}^{\mbox{\tiny 1}}, \ldots, n_{i,j,s}^{\mbox{\tiny N}^\prime} $ , for all $s \in S, i \in I_s, j \in J_s, n^\prime=1,\ldots, 10,000$.  We follow the same standard statistical method as in prior distributionally robust literature to design the parameters of the joint Uniform distribution with varying levels of correlations, while keeping the mean and support of the $N^\prime$ out-of-sample data the same as those of the $N$ in-sample data.
\end{enumerate}
\color{black}

We evaluate the out-of-sample performance of the optimal \blue{SP} and DR solutions as follows. First, we fix the optimal first-stage allocation decisions ($m_s^*$, for all $s \in S$) in the \blue{SP} model. Then we simulate the second-stage recourse problem with dynamic routing using the $N^\prime$ data to compute passenger waiting and riding time costs.

\vspace{1mm}

\blue{Table~\ref{table:InSample} presents the means and quantiles of the total cost (TC), second-stage cost (2nd-stage), total waiting time per region (TWT), and total riding time per region (TRT) yielded by the optimal solutions of the DR and SP for insatnces 1--3 under perfect distributional information (i.e., LogN distribution).} \blue{Table~\ref{table:InSampleNYC} presents the results for the NYC instance}. Clearly, by allocating more vehicles in each region, the DR results in a higher vehicle fixed (one-time) cost, and thus a higher total cost, than the \blue{SP}. However, the DR also results a in significantly lower second-stage cost and, in particular, substantially lower waiting time on average and at all quantiles, and hence offers a better quality of service and greater passenger satisfaction. For example, consider instance 2 and 

\vspace{1mm}

\begin{landscape}
\begin{table}[t!]
\center 
 \footnotesize
   \renewcommand{\arraystretch}{0.3}
\caption{In-sample performance of optimal sizing and allocation decisions for Inst 1--3 under perfect distributional information, $f_s=4000$.} 
\begin{tabular}{cclllllllllllllllllllllll}
\hline
                  & 						&							&							&  & &\multicolumn{4}{c}{\textbf{TWT for each region}}	& & 	\multicolumn{5}{c}{\textbf{TRT for each region}} \\ \cline{7-11} \cline{13-17}
\textbf{Inst}&\textbf{R}& \textbf{Metric}  &\textbf{Model}  & \textbf{TC} & \textbf{2nd-stage}&  1  & 2 & 3 & 4 & 5 & & 1 & 2 & 3 & 4 & 5\\
\hline
1 & [1,4]& Mean & DR & 117,944 & 5,944  &  354 & 68 & 104 & 111 &&& 708 & 1,043 & 1,129 & 1,793 \\
  & 		& &  \blue{SP} & 88,097 & 24,097 &2,792  & 2,102 & 2,007 & 2,826  & && 604 & 1,061& 1,140&  1,837\\
  \\
 	&      & Median & DR & 117,836 & 5,836 & 320& 50& 100& 110 &&& 708 & 1,045& 1,128& 1,795    \\
 	&		&				& \blue{SP}& 87,866 & 23,866 & 2,770& 2,090 & 1,980 & 2,770 & &&604 & 1,066 &1,140& 1,836 \\
 	\\
 	&    & 75\%-q & DR & 118,275 & 6,275 & 400& 80& 120& 130 & && 732& 1,077& 1,164& 1,842 \\
 	& 	 &  						& \blue{SP} & 90,959 & 27,959 & 3,110&  2,420& 2,340& 3,230 & && 614 & 1,086 & 1,174& 1,885\\
 	\\
	&    & 95\%-q & DR &  119,440 & 7,440 & 670 & 190& 170 & 180 & && 760& 1,121& 1,218& 1,921 \\
	& 	& 							& \blue{SP} & 95,817 & 31,817 & 3,610& 2,980& 2,960&  3,910 & && 625 & 1,110& 1,213 & 1,949\\
	\\
	& [3, 7] & Mean & DR  &245,399&	9,399&   26 & 19 & 50 & 9 &&& 1,779 & 2,155 & 2,171 & 3,087\\
	&			& 			& \blue{SP}  & 176,342	&32,342  & 1,681& 2,686&  3,238 & 3,910& & & 1,811 & 2,162& 2,211& 3,129\\
	\\
 	& & Median & DR &245,353	& 9,353 & 20 & 10 & 40 & 10 & && 1,777& 2,156& 2,171& 3,089 \\
 	&  &           & \blue{SP} & 175,619	&31,619	& 1,600& 2,580& 3,150& 3,820 &&& 1,809& 2,164& 2,273 & 3,134\\
 \\
  	&    & 75\%-q & DR & 245,794	& 9,794 & 40 & 30 & 70 & 10  && & 1,844 & 2,232& 2,244 & 3,174\\	
  	&	 & 			 & \blue{SP} & 181,013	&37,013& 2,010 & 3,280& 3,820& 4,600 & && 1,879& 2,236& 2,273& 3,205\\
  	\\
  	&    & 95\%-q & DR &  246,447	&10,447 & 70 & 60 & 110 & 30 &&& 1,932& 2,321& 2,359& 3,295\\
    	&	 & 			 & \blue{SP} & 189,398	 & 45,398 & 2,700& 4,200& 4,890& 5,950 && & 1,974& 2,313& 2,355& 3,276 \\		  	
    	\\\
 2 & [1,4] & Mean &  DR &94,617 &	6,617& 613 & 139 & 215 & 301 && & 532& 921& 1,211& 1,418 \\   
  &			&         & \blue{SP} & 78,238	 &10,238 & 485 & 1,098&  348&  1,083&&& 514 & 968& 1,271& 1,454\\
  \\
  & 			& Median & DR & 94,421	& 6,421& 590 & 100& 200& 280&&& 530& 922&1,209& 1,420\\
  &			& 				& \blue{SP} & 78,049&	10,049 & 470 & 1,070 & 330 & 1050 && & 514& 968& 1,268& 1,459 \\
  \\
  &    & 75\%-q & DR & 95,137	& 7,137 & 680 & 170& 260& 350 &&& 551 & 947& 1,250& 1,469 \\
  &   &				& \blue{SP} & 79,282& 	11,282 & 540 & 1,290& 400& 1,240 &&& 532 & 995& 1,313& 1,502\\ 
  \\
  	&    & 95\%-q &  DR  & 96,619	& 8,619 & 860& 390& 390& 460 && & 583 & 979& 1,317& 1,540\\   
  	&		&				& \blue{SP} & 81,439	& 13,439&  680 & 1,650 & 570& 1,550 & && 557 & 1,034 & 1,384&  1564\\
  	\\
  	& [3,7] & Mean & DR & 207,918 &	7,918 & 175 & 84 & 97 & 32 & &&  1,299 & 1,500 & 2,100 & 2,241 \\
  	&			& 			& \blue{SP} & 144,785 & 	3,278& 1,194 & 3,357 & 3,070 & 5,123 && & 1,362 & 1,625 & 2,176 & 2,134 \\
  	\\
  	& & Median & DR &207,836	&7,836 & 220 & 70 & 90 & 30 &&& 1,300 &   1,498 & 2,101 & 2,237\\
  	& & 				&\blue{SP} & 144,200	& 32,200 & 1,160 & 3,290 & 2,970 & 5,030 && & 1,358& 1,628 & 2,177& 2,137\\
  	\\
  	& & 75\%-q &  DR & 208,441 & 	8,441 & 220& 120 & 130 & 50 &&& 1,351& 1,553& 2,171& 2,326 && \\ 
  	& &   		& \blue{SP} & 149,331	& 	37,331 & 1,490 & 3,890 & 3,620& 5,920 & && 1,416 & 1,664 & 2,239 & 2,172\\
  	\\
  	& & 95\%-q & DR & 209,384	 & 9,384 & 320 & 200 & 200 & 80 & && 1424 & 1,623&  2,302& 2,435\\
  	& & 				& \blue{SP}& 157,029	&45,029 & 2,000 & 4,760&4,710& 7,170 &&& 1,490& 1,714& 2,318 & 2,227\\
  	\\
  	\\
  3 & [1,4] & Mean & DR & 	132,669 & 	8,669 & 853 & 104 & 125 & 232 & 190 & & 475 & 953 & 1,229 & 1,405 & 1,596 \\
  & & 		&   \blue{SP}  & 102,232 & 	22,232 & 513 & 1,004 & 1,211 & 2,751 & 2,752  & & 522 & 1,000 & 1,326 & 1,320 & 1,599 \\
  \\
  & & Median& DR & 132,460	& 8,460 & 840 & 80 & 90 & 210 & 190 & & 476 & 955 & 1,226 & 1,407 & 1,596  \\
   &  &  & \blue{SP} & 101,955&	21,955& 510 & 980 & 1,170 & 2,730 & 2,700 & &522 & 1,001 & 1,328 & 1,320 & 1,604\\
   \\ 
  & &  75\%-q &  DR & 133,344	& 9,344 & 940 & 140 & 170 & 280 & 220 & & 492 & 983 & 1,272 & 1,450& 1,647\\ 
  & &   			& \blue{SP} & 104,734	& 24,734 &  580 & 1,180 & 1,420 & 3,130 & 3,100 & & 541 & 1,026 & 1,372 & 1,344 & 1,631 \\
  \\
&  & 95\%-q  & DR & 135,111 &	11,111 & 1,110 & 270 & 370& 440 & 310 && 517& 1,020& 1,339 &1,524 &1,711 \\
& & 				& \blue{SP} & 109,032&	29,032 & 720& 1,510& 1,830 & 3,700 & 3,700 && 570& 1,067& 1,426 & 1,381& 1,668\\
\\
& [3, 7] & Mean & DR & 252,442	& 12,442 & 84 & 107 & 96 & 184 & 648  && 1,299 & 1,571 & 2,221 & 2,293 & 3,155 \\
&			& 		& \blue{SP} &203,413	& 35,413 & 1,005 & 1,683& 3,284 &3,530 & 2,940  && 1,366& 1,677 & 2,154& 2,136& 3,195\\
\\ 
& 			&  Median &DR &  252,740	& 12,740 & 80 & 100 & 80 & 130 & 600 && 1,299 & 1,571 & 2,221 & 2,292 & 3,154 \\
&			& 			& \blue{SP} & 202,673	& 34,673 & 940& 1,610& 3,200& 3,440& 2,880& & 1,363& 1,679& 2,158& 2,136& 3,197 \\
\\
  & &  75\%-q  & DR & 254,142	 & 14,142 & 110 & 140 & 120 & 190 & 790 && 1,348 &  1,623 & 2,293& 2,375 & 3,243\\
  & & 			& \blue{SP} & 208,480	& 40,480 & 1,270& 2,060& 3,880& 4,180& 3,440& & 1,419& 1,732& 2,215& 2,179& 3,275\\
  \\
  & & 95\%-q & DR & 253,026 &	13,026 & 170 & 220& 250& 340& 1180& & 1,425& 1,701& 2,417& 2,499& 3,364\\
  &&				&	 \blue{SP} & 217,696	& 49,696& 1,850& 2,790& 4,940 & 5,300& 4,360&  & 1,497& 1,805& 2,283& 2,248& 3,383\\
\hline
\end{tabular}\label{table:InSample}
\end{table}
\end{landscape}

\begin{table}[t!]
\center 
 \footnotesize
 \color{black}
   \renewcommand{\arraystretch}{0.3}
\caption{\blue{In-sample performance of optimal sizing and allocation decisions for the NYC instance under perfect distributional information, $f_s=4000$.}} 
\begin{tabular}{cclllllllllllllllllllllll}
\hline
               & & & & \multicolumn{4}{c}{\textbf{TWT for each region}}	& & 	\multicolumn{4}{c}{\textbf{TRT for each region}} \\ \cline{5-8} \cline{10-14}
 \textbf{Metric}  &\textbf{Model}  & \textbf{TC} & \textbf{2nd-stage}& 1 & 2 & 3 & 4  &&  1& 2 & 3 & 4 \\
 \hline 
Mean	&	DR	&	86,018	&	2,018	&	142	&	92	&	31	&	16	&	&	223	&	473	&	544	&	212	\\
	&	SP	&	50,210	&	10,210	&	817	&	1,243	&	1,445	&	658	&	&	325	&	628	&	630	&	296	&\\
	\\
Median	&	DR	&	85,982	&	1,982	&	140	&	90	&	30	&	10	&	&	221	&	472	&	538	&	211	\\
	&	SP	&	49,341	&	9,341	&	730	&	1,160	&	1,270	&	580	&	&	319	&	625	&	625	&	292	\\
	\\
75\%-q	&	DR	&	86,259	&	2,259	&	170	&	110	&	40	&	30	&	&	224	&	507	&	591	&	237	\\
	&	SP	&	52,374&	12,374	&	1,010	&	1,510	&	1,810	&	820	&	&	367	&	691	&	681	&	335	\\
	\\
95\%q	&	DR	&	86,727	&	2,727	&	210	&	150	&	70	&	50	&	&	278	&	549	&	662	&	278	\\
	&	SP	&	58,652&	18,652	&	1,590	&	2,170	&	3,050	&	1310	&	&	443	&	788	&	788	&	393	\\
\hline
\end{tabular}\label{table:InSampleNYC}
\end{table}

\noindent range [1,4]. By allocating 22 and 17 vehicles, respectively, the DR and \blue{SP} result in 88,000 and 68,000 vehicle fixed costs. However, the average second-stage cost and total waiting time (overall regions) of the DR are 55\% and 137\% lower than those of the \blue{SP}, respectively. It is not surprising that the total cost of the objective function in \blue{SP} is lower than that in DR, since we assume perfect information about the exact demand distribution in this case.

\color{black}
Table \ref{table:OutofSample} presents the means and quantiles of the total and second stage costs from the optimal solutions of DR and \blue{SP} under misspecified distributional information. For total cost, the out-of-sample results in Table \ref{table:OutofSample} shows that there is still no clear winner \blue{for instances 1--3}, while DR has significantly lower costs in some worst-case instances, for example in instances 2 with larger demands of $R = [3,7]$ (average TC of DR and \blue{SP} are \$215,208  and \$226,406, respectively). For second-stage costs, DR out-performs \blue{SP} significantly, since DR is designed to be robust against worst-case demand distributions \noindent in the second stage. \blue{Interestingly, the DR model provides significantly lower total and second-stage costs for the NYC instance compared to that of the SP model, which may be due to the much stronger demand uncertainty in the NYC instance constructed using real data. These simulation results demonstrate the value of incorporating both uncertainty and ambiguity into fleet sizing, allocation, routing, and scheduling models. }

\color{black}

\subsection{\textbf{Sensitivity Analysis}}\label{sec5.4:sensitivity}
\noindent \textcolor{black}{In this section, we study the sensitivity of DR and \blue{SP} solutions to various input parameter settings. For illustrative purposes and presentation brevity, we consider instance 1 for this experiment (we observe similar results for instances 2-4 and the NYC instance). For each experiment, we simulate the optimal solutions of the two-stage DR and two-stage \blue{SP} (called TSM henceforth) under a sample of 10,000 scenarios of the number of passengers demanding last-mile service.}

\vspace{1mm}
\begin{table}[t!]
  \centering
  \footnotesize
  \caption{Out-of-sample performance of optimal sizing and allocation decisions under misspecified distribution, $f_s=4000$.}
  \hspace*{-1cm}
    \begin{tabular}{rrrlrrrrrrrr}
    \toprule
    \multicolumn{1}{l}{\textbf{Instance}} & \multicolumn{1}{l}{\textbf{R}} & \multicolumn{1}{l}{\textbf{Metric}} & \textbf{Model} & \multicolumn{1}{l}{\textbf{TC}} & \multicolumn{1}{l}{\textbf{2nd-Stage}} & \multicolumn{1}{c}{\textbf{Instance}} & \multicolumn{1}{l}{\textbf{R}} & \multicolumn{1}{l}{\textbf{Metric}} & \multicolumn{1}{l}{\textbf{Model}} & \multicolumn{1}{l}{\textbf{TC}} & \multicolumn{1}{l}{\textbf{2nd-Stage}} \\
    \midrule
    1     & \multicolumn{1}{l}{[1,4]} & \multicolumn{1}{l}{Mean} & DR    & 126,262 & 14,262 & 3     & \multicolumn{1}{l}{[1,4]} & \multicolumn{1}{l}{Mean} & \multicolumn{1}{l}{DR} & 144,217 & 20,217 \\
          &       &       & \blue{SP}   & 116,912 & 52,912 &       &       &       & \multicolumn{1}{l}{\blue{SP}} & 139,367 & 43,367 \\
          &       & \multicolumn{1}{l}{Median} & DR    & 123,791 & 11,791 &       &       & \multicolumn{1}{l}{Median} & \multicolumn{1}{l}{DR} & 148,987 & 24,987 \\
          &       &       & \blue{SP}   & 119,346 & 55,346 &       &       &       & \multicolumn{1}{l}{\blue{SP}} & 148,516 & 52,515 \\
          &       & \multicolumn{1}{l}{75\%-q} & DR    & 138,444 & 26,444 &       &       & \multicolumn{1}{l}{75\%-q} & \multicolumn{1}{l}{DR} & 156,193 & 32,193 \\
          &       &       & \blue{SP}   & 145,899 & 81,899 &       &       &       & \multicolumn{1}{l}{\blue{SP}} & 162,518 & 66,518 \\
          &       & \multicolumn{1}{l}{95\%-q} & DR    & 138,444 & 26,444 &       &       & \multicolumn{1}{l}{95\%-q} & \multicolumn{1}{l}{DR} & 156,193 & 32,193 \\
          &       &       & \blue{SP}   & 145,899 & 81,899 &       &       &       & \multicolumn{1}{l}{\blue{SP}} & 162,518 & 66,518 \\
          & \multicolumn{1}{l}{[3,7]} & \multicolumn{1}{l}{Mean} & DR    & 257,150 & 21,150 &       & \multicolumn{1}{l}{[3,7]} & \multicolumn{1}{l}{Mean} & \multicolumn{1}{l}{DR} & 282,210 & 42,210 \\
          &       &       & \blue{SP}   & 216,848 & 68,848 &       &       &       & \multicolumn{1}{l}{\blue{SP}} & 314,953 & 146,953 \\
          &       & \multicolumn{1}{l}{Median} & DR    & 252,486 & 16,486 &       &       & \multicolumn{1}{l}{Median} & \multicolumn{1}{l}{DR} & 272,873 & 32,873 \\
          &       &       & \blue{SP}   & 210,143 & 62,143 &       &       &       & \multicolumn{1}{l}{\blue{SP}} & 293,358 & 125,358 \\
          &       & \multicolumn{1}{l}{75\%-q} & DR    & 272,729 & 36,729 &       &       & \multicolumn{1}{l}{75\%-q} & \multicolumn{1}{l}{DR} & 321,239 & 81,239 \\
          &       &       & \blue{SP}   & 254,225 & 106,225 &       &       &       & \multicolumn{1}{l}{\blue{SP}} & 441,640 & 273,640 \\
          &       & \multicolumn{1}{l}{95\%-q} & DR    & 277,311 & 41,311 &       &       & \multicolumn{1}{l}{95\%-q} & \multicolumn{1}{l}{DR} & 329,795 & 89,795 \\
          &       &       & \blue{SP}   & 269,561 & 121,561 &       &       &       & \multicolumn{1}{l}{\blue{SP}} & 461,111 & 293,111 \\
    2     & \multicolumn{1}{l}{[1,4]} & \multicolumn{1}{l}{Mean} & DR    & 102,883 & 14,883 &     \multicolumn{1}{r}{\blue{NYC}} &       &    \multicolumn{1}{l}{\blue{Mean} }  &  \multicolumn{1}{l}{\blue{DR}}     & \blue{87,093}       & \blue{3,093}\\
          &       &       & \blue{SP}   & 94,451 & 26,451 &       &       &     & \multicolumn{1}{l}{\blue{SP}} & \blue{147,923}      & \blue{107,923} \\
          &       & \multicolumn{1}{l}{Median} & DR    & 106,262 & 18,262 &       &      &  \multicolumn{1}{l}{\blue{Median}}      &   \multicolumn{1}{l}{\blue{DR}}     & \blue{87,165}      &  \blue{3,165}\\
          &       &       & \blue{SP}   & 101,276 & 33,276 &       &       &       &   \multicolumn{1}{l}{\blue{SP}}     &   \blue{114,264}    & \blue{103,443} \\
          &       & \multicolumn{1}{l}{75\%-q} & DR    & 110,987 & 22,987 &       &       &  \multicolumn{1}{l}{\blue{75\%-q}}     &    \multicolumn{1}{l}{\blue{DR}}    & \blue{88,210}      & \blue{4,210} \\
          &       &       & \blue{SP}   & 108,761 & 40,762 &       &       & &    \multicolumn{1}{l}{\blue{SP}}       & \blue{182,414} &\blue{168,191}    \\
          &       & \multicolumn{1}{l}{95\%-q} & DR    & 110,987 & 22,987 &       &       &    \multicolumn{1}{l}{\blue{95\%-q}}    &     \multicolumn{1}{l}{\blue{DR}}   &   \blue{90,838}    &  \blue{6,838}\\
          &       &       & \blue{SP}   & 108,761 & 40,762 &       &     && \multicolumn{1}{l}{\blue{SP}}  &    \blue{26,2214}   &       \blue{222,791}        \\
          & \multicolumn{1}{l}{[3,7]} & \multicolumn{1}{l}{Mean} & DR    & 215,208 & 15,208 &       &       &       &       &       &  \\
          &       &       & \blue{SP}   & 226,406 & 110,406 &       &       &       &       &       &  \\
          &       & \multicolumn{1}{l}{Median} & DR    & 210,996 & 10,995 &       &       &       &       &       &  \\
          &       &       & \blue{SP}   & 223,017 & 107,017 &       &       &       &       &       &  \\
          &       & \multicolumn{1}{l}{75\%-q} & DR    & 230,932 & 30,932 &       &       &       &       &       &  \\
          &       &       & \blue{SP}   & 298,678 & 182,678 &       &       &       &       &       &  \\
          &       & \multicolumn{1}{l}{95\%-q} & DR    & 233,606 & 33,605 &       &       &       &       &       &  \\
          &       &       & \blue{SP}   & 313,528 & 197,528 &       &       &       &       &       &  \\
          \bottomrule
    \end{tabular}%
  \label{table:OutofSample}%
\end{table}%

\noindent \textbf{Impact of variability in demand/demand ranges}
\vspace{1mm}

\textcolor{black}{First, we analyze the DR and \blue{SP} solutions' sensitivity to the variability and volume of the number of passengers arriving at each service region with each train. In addition to the base demand range (Range 1, [1,4]), we consider four additional ranges: [1,6], [1,8], [4,7], and [6,9]. In [1,6] and [1,8],  we increase the variability (difference between the lower and upper bounds) of the number of $n_{i,j,s}$ from 3 to 5 and 7, respectively. In [4,7] and [6,9], we keep the difference between the upper and lower bounds of $n_{i,j,s}$ as in the base range (i.e., 3), and increase the demand volume (lower and upper bounds) to [4,7] and [6, 9]. We keep cost parameters as in the base case settings, i.e., $f_s=4,000$, $\beta_w=2$, and $\beta_r=1$.}

\vspace{1mm}

Figure~\ref{Fig:Sens1} presents the optimal fleet size (i.e., $\sum \limits_{s \in S} m_s^*$) and the average second-stage cost (waiting+riding time costs) as a function of the demand range. It is quite apparent from Figure~\ref{Fig:Sens1} that both models tend to allocate more vehicles under higher variability and volume of demand. By allocating more vehicles, the DR mitigates the increase in passengers' variability and volume by maintaining significantly lower waiting and riding time costs.

      \begin{figure}
     \begin{subfigure}[b]{0.5\textwidth}
 \centering
        \includegraphics[width=\textwidth]{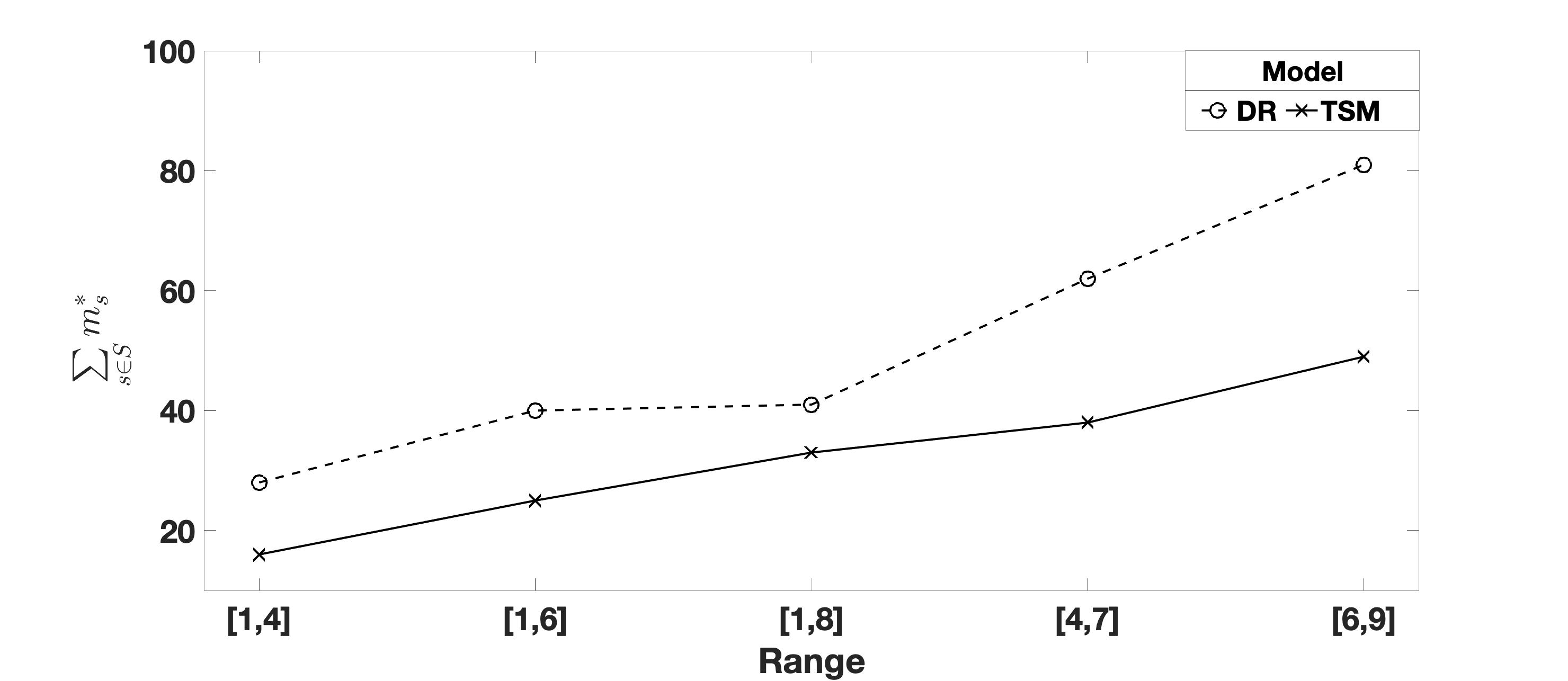}
        \caption{Optimal fleet size vs. Demand range}
        \label{Fig1a:Q_open}
    \end{subfigure}%
    \begin{subfigure}[b]{0.5\textwidth}
            \includegraphics[width=\textwidth]{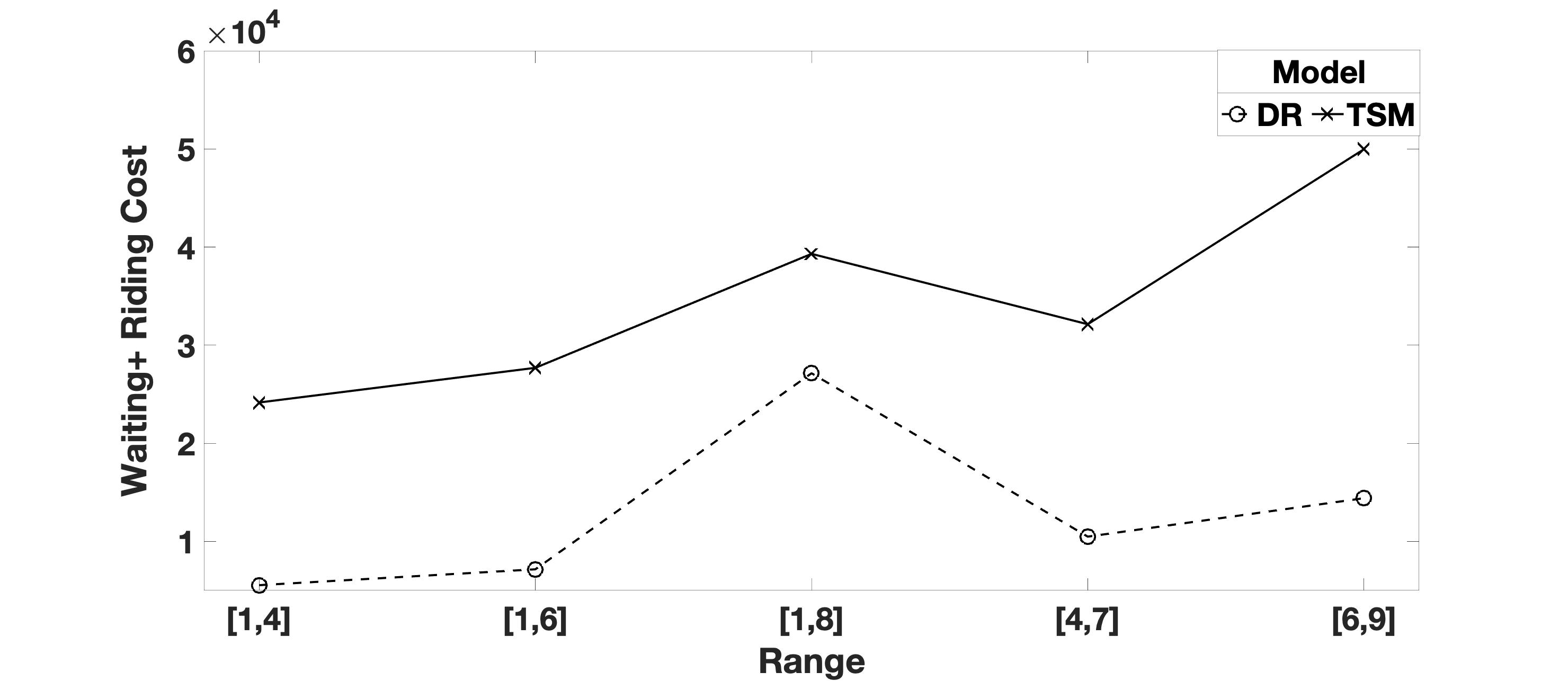}
      \caption{Average 2nd-stage cost vs. Demand range}
      \label{Fig1b:Q_total}
    \end{subfigure}%
        \caption{Optimal fleet size ($\sum \limits_{s \in S} m_s^*$) and the associated second-stage cost (waiting+riding time costs) under different demand range. TSM is the two-stage SP model }\label{Fig:Sens1}
    \end{figure}
    

\vspace{1mm}

\noindent \textbf{Impact of cost parameters}
\vspace{1mm}

Next, we analyze the DR and \blue{SP} solutions' sensitivity to the cost parameters. We fix the demand range to [1,4] and [4,7] as examples of a low and relatively high volume of passengers, and vary $f_s \in \{ 4,000, 7,000, 10,000 \} $ and $(\beta_w, \beta_r) \in \{ (1,2), (8, 4), (32, 16)  \}$.  Figures~\ref{Fig:Sens2} and \ref{Fig:Sens3} present the optimal fleet size, $\sum \limits_{s \in S}m_s^*$, and the associated second-stage (waiting+riding time costs) cost as a function of $f_s$ and $(\beta_w, \beta_r) $ under demand range $[1,4]$ and $[1,7]$, respectively. 

\vspace{1mm}

 \begin{figure}[t!]
     \begin{subfigure}[b]{0.5\textwidth}
 \centering
        \includegraphics[width=\textwidth]{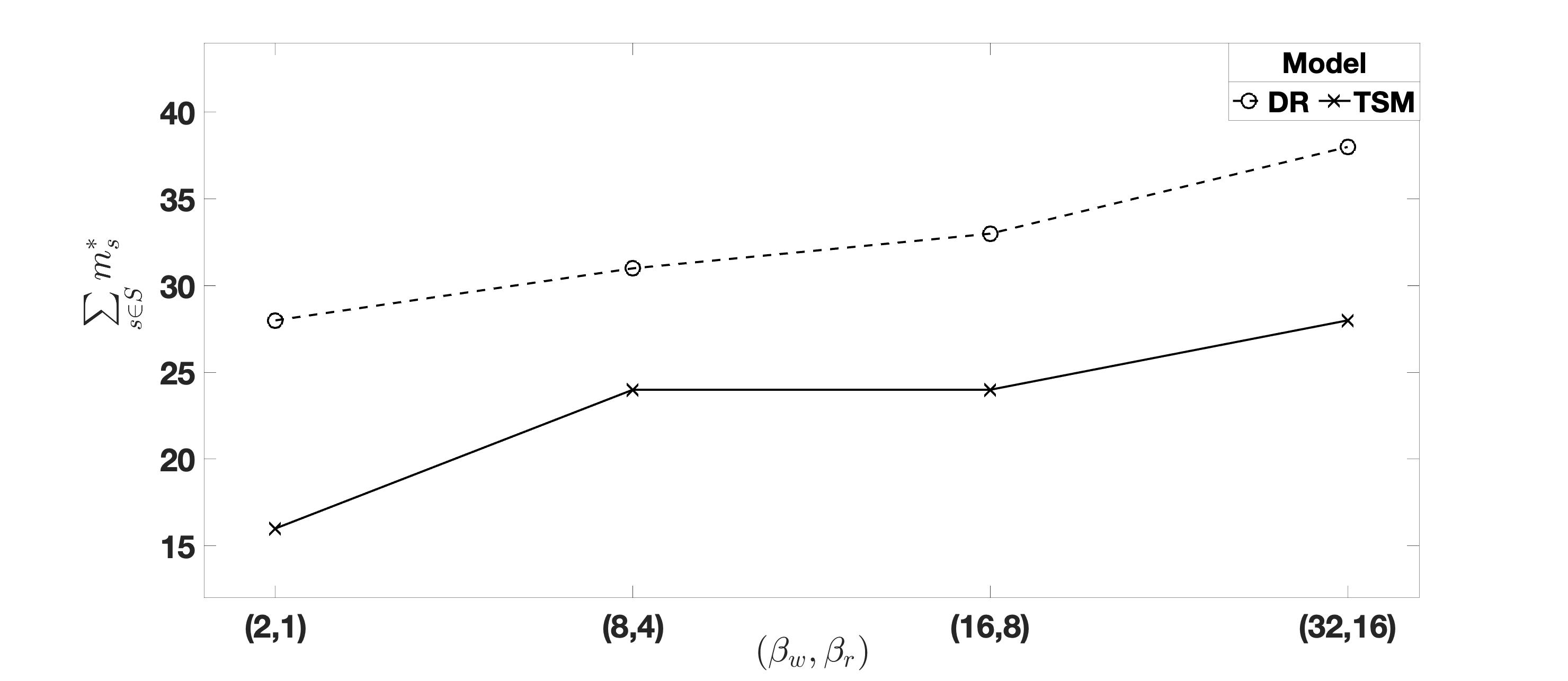}
        \caption{Optimal fleet size  vs. ($\beta_w, \beta_r$), $f_s=4,000$}
        \label{Fig1a:M_4000_1_4}
    \end{subfigure}%
    \begin{subfigure}[b]{0.5\textwidth}
            \includegraphics[width=\textwidth]{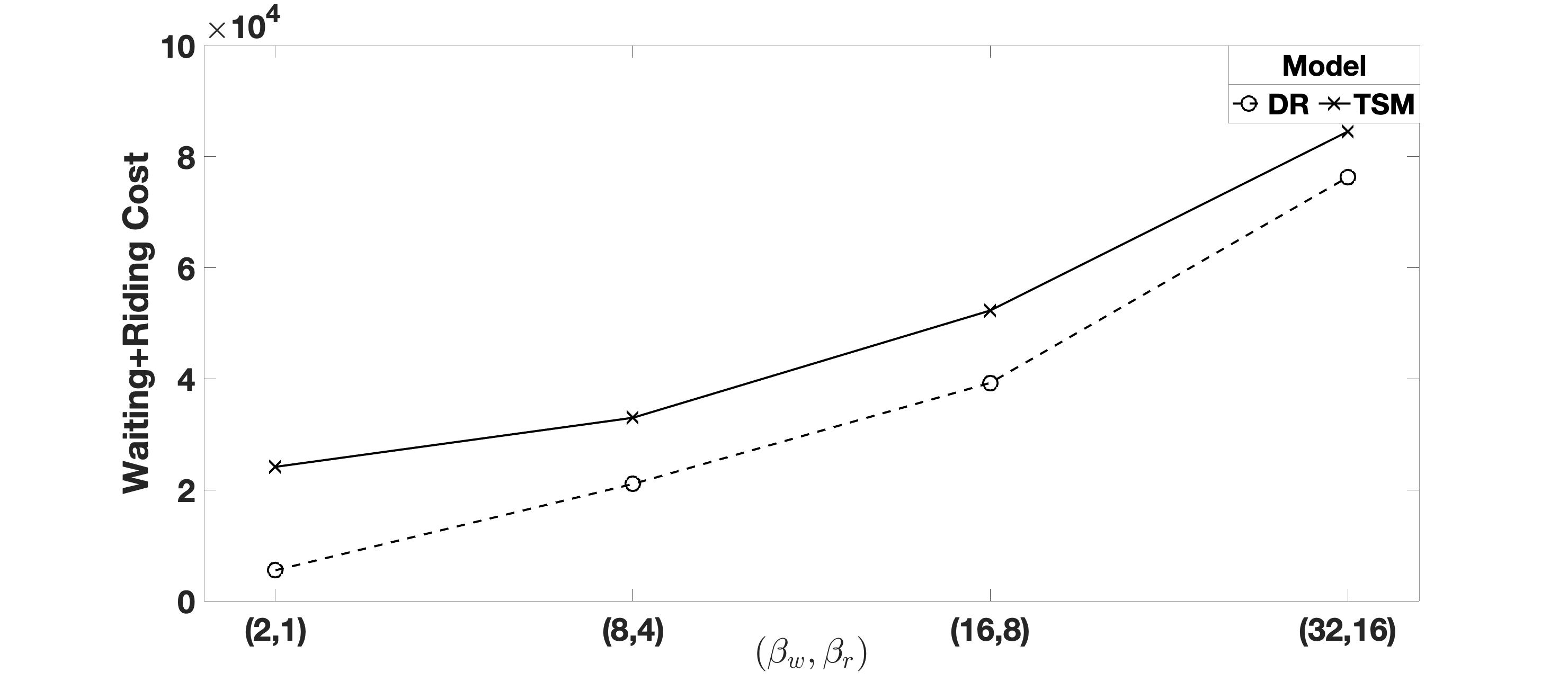}
      \caption{Average 2nd-stage cost vs. ($\beta_w, \beta_r$), $f_s=4,000$}
      \label{Fig1b:Cost_40000_1_4}
    \end{subfigure}%
    
     \begin{subfigure}[b]{0.5\textwidth}
 \centering
        \includegraphics[width=\textwidth]{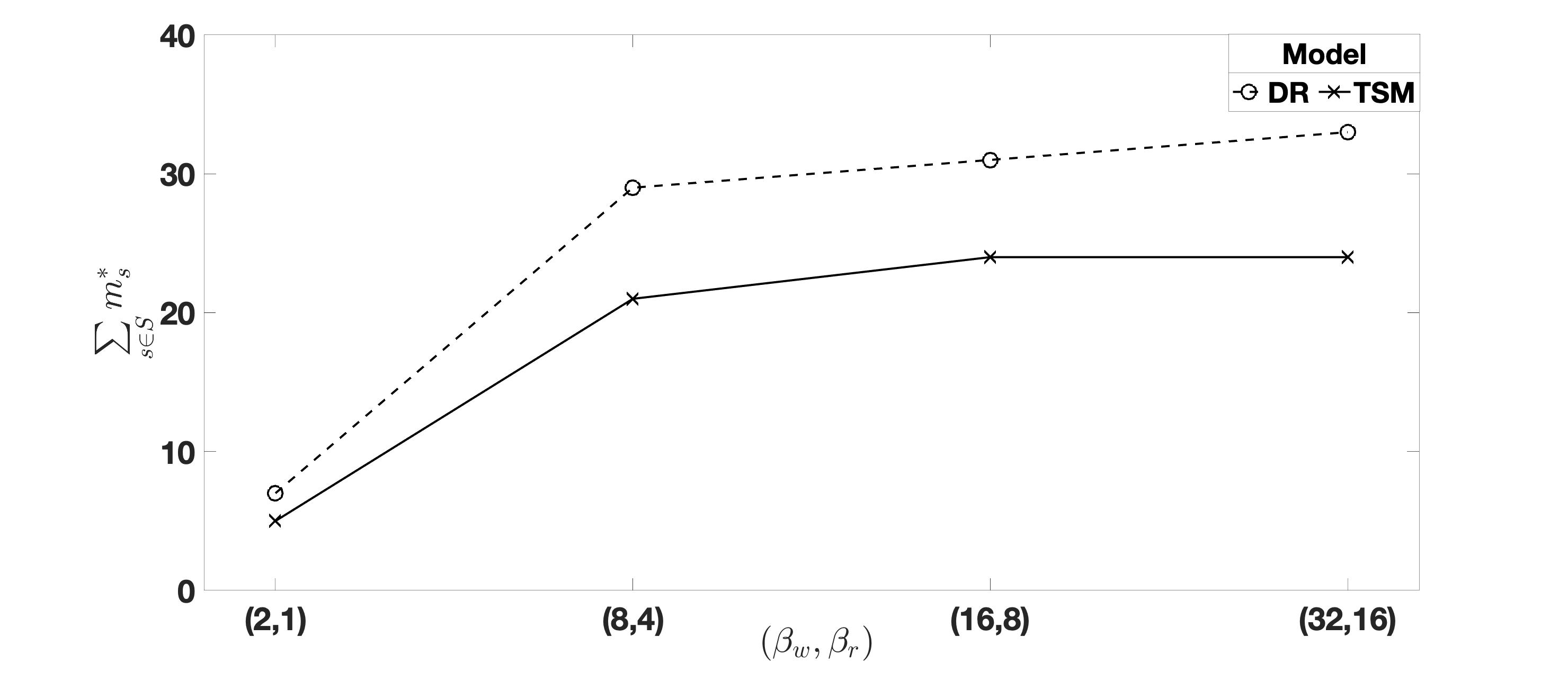}
        \caption{Optimal fleet size vs. ($\beta_w, \beta_r$), $f_s=7,000$}
        \label{Fig1a:M_7000_1_4}
    \end{subfigure}%
    \begin{subfigure}[b]{0.5\textwidth}
            \includegraphics[width=\textwidth]{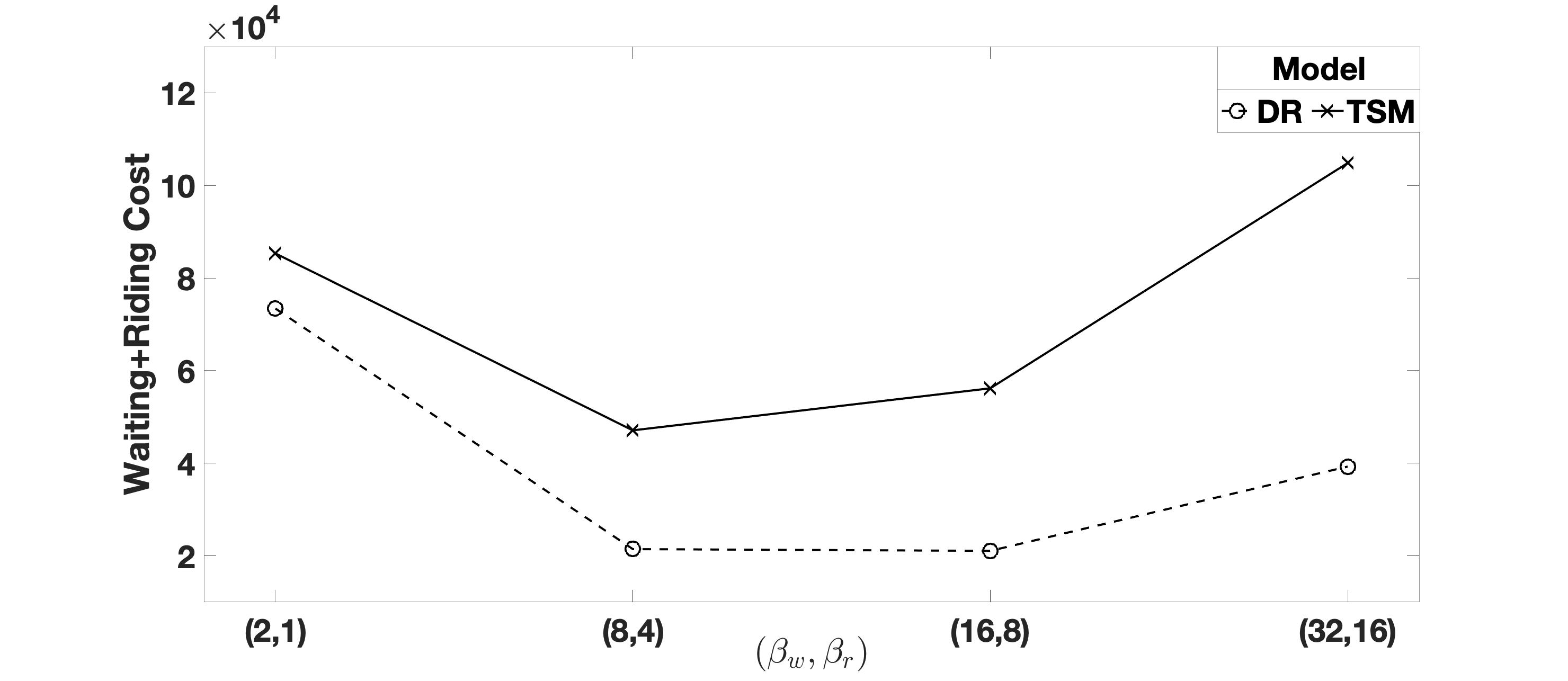}
      \caption{Average 2nd-stage cost vs.  ($\beta_w, \beta_r$), $f_s=7,000$}
      \label{Fig1b:Cost_70000_1_4}
    \end{subfigure}%
    
\begin{subfigure}[b]{0.5\textwidth}
 \centering
        \includegraphics[width=\textwidth]{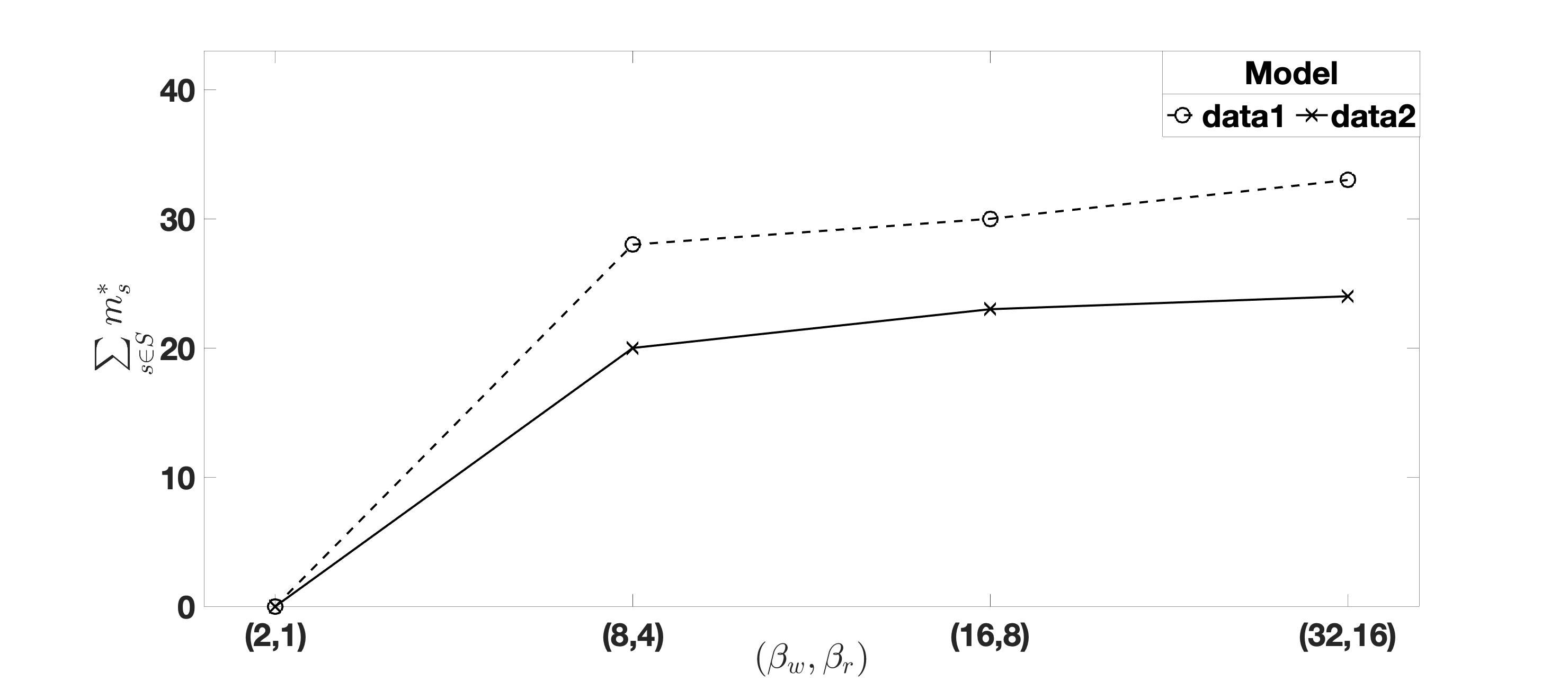}
        \caption{Optimal fleet size vs. ($\beta_w, \beta_r$), $f_s=10,000$}
        \label{Fig1a:M_10000_1_4}
    \end{subfigure}%
    \begin{subfigure}[b]{0.5\textwidth}
            \includegraphics[width=\textwidth]{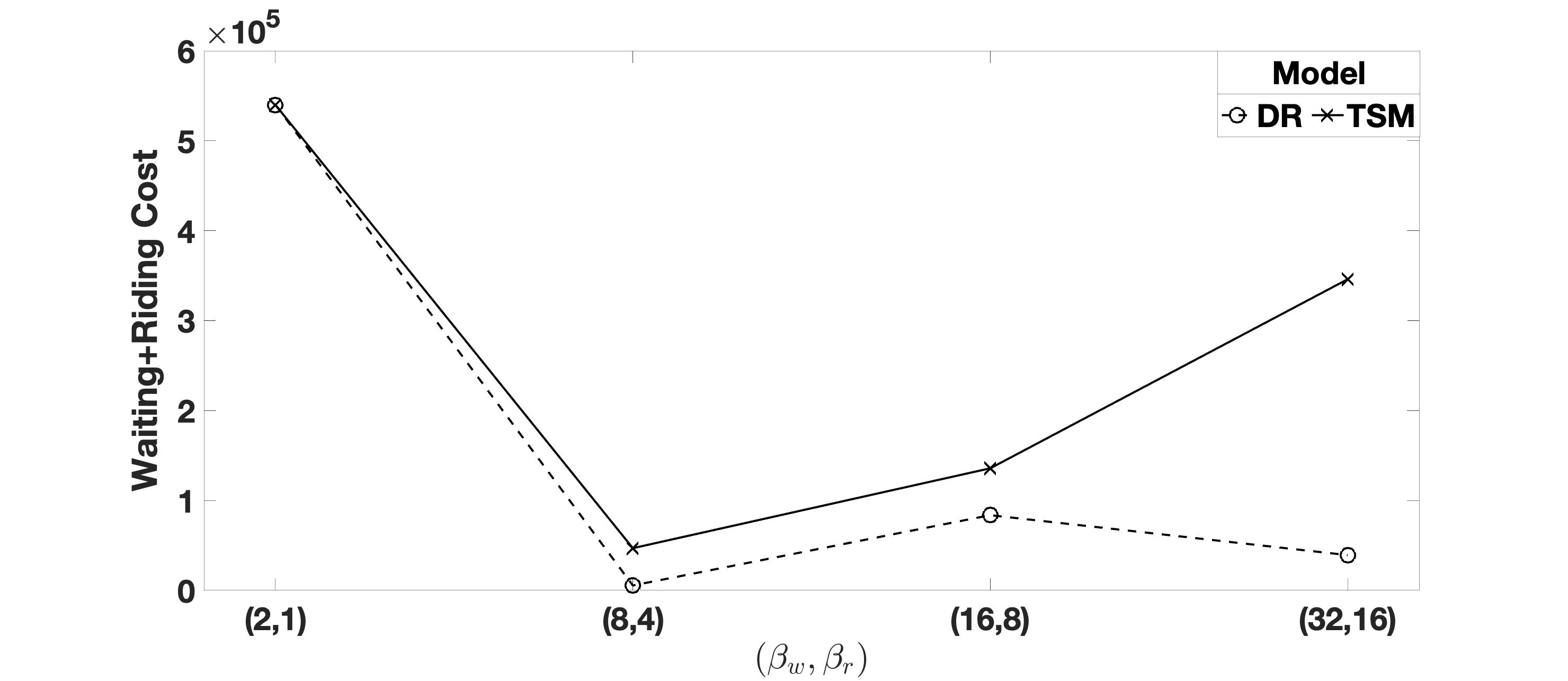}
      \caption{Average 2nd-stage cost vs.  ($\beta_w, \beta_r$), $f_s=10,000$}
      \label{Fig1b:Cost_100000_1_4}
    \end{subfigure}%
    
        \caption{Optimal fleet size ($\sum \limits_{s \in S} m_s^*$) and the associated second-stage cost (waiting+riding time costs) under demand range $[1,4]$ and different values of $f_s$ and ($\beta_w, \beta_r$). TSM is the two-stage SP model.}\label{Fig:Sens2}
    \end{figure}

\textcolor{black}{We first observe that the optimal fleet size, $\sum \limits_{s \in S}m_s^*$ ,  yielded by the DR model is always larger than that of the \blue{SP} model under all values of $f_s$ and ($\beta_w, \beta_r$). Second, we observe that both models allocate more vehicles under [4,7] than under [1,4], irrespective of the values of cost parameters, which makes sense given that the former implies a higher volume of demand for last-mile service that needs fulfilling.} \textcolor{black}{Third, we observe that for a fixed value of ($\beta_w, \beta_r$), the optimal fleet size $\sum \limits_{s \in S}m_s^*$ of the DR and \blue{SP} decreases as the unit fixed cost $f_s$ increases. For example, when  ($\beta_w, \beta_r$)= (2,1) and range equal [4,7], the $\sum \limits_{s \in S}m_s^*$ of the DR and \blue{SP}, respectively, decreases from 62 and 39 to 12 and 7 vehicles as $f_s$ increases from $f_s=4,000$ to $f_s=7,000$ (see Figure~\ref{Fig1a:M_4000_4_7} and \ref{Fig1a:M_7000_4_7}). }

\vspace{1mm}

\textcolor{black}{Fourth, we observe that as the values of ($\beta_w, \beta_r$) increase (i.e., passenger waiting and riding time become more important/expensive), the optimal fleet size $\sum \limits_{s \in S}m_s^*$ yielded by DR and \blue{SP} increase regardless of the unit fixed cost $f_s$ and range of passengers. Finally, we observe that by allocating a larger fleet, the DR always yields a substantially lower second-stage cost (i.e., waiting and riding times) than the \blue{SP}, which indicates a better quality of service for passengers. However, this of course comes at a higher fixed cost. The relative difference in the fixed cost between DR and \blue{SP} ranges from 0 to 40\%, and the relative difference in the second-stage cost ($\frac{DR-\blue{SP}}{DR} \times 100\%$) from 15\% to 333\%. Practitioners may be willing to pay the extra one-time fixed cost of DR solutions to provide a better quality of service in terms of lower waiting and riding times, and thus maintain customer satisfaction and a good business reputation. }

\begin{figure}
     \begin{subfigure}[b]{0.5\textwidth}
 \centering
        \includegraphics[width=\textwidth]{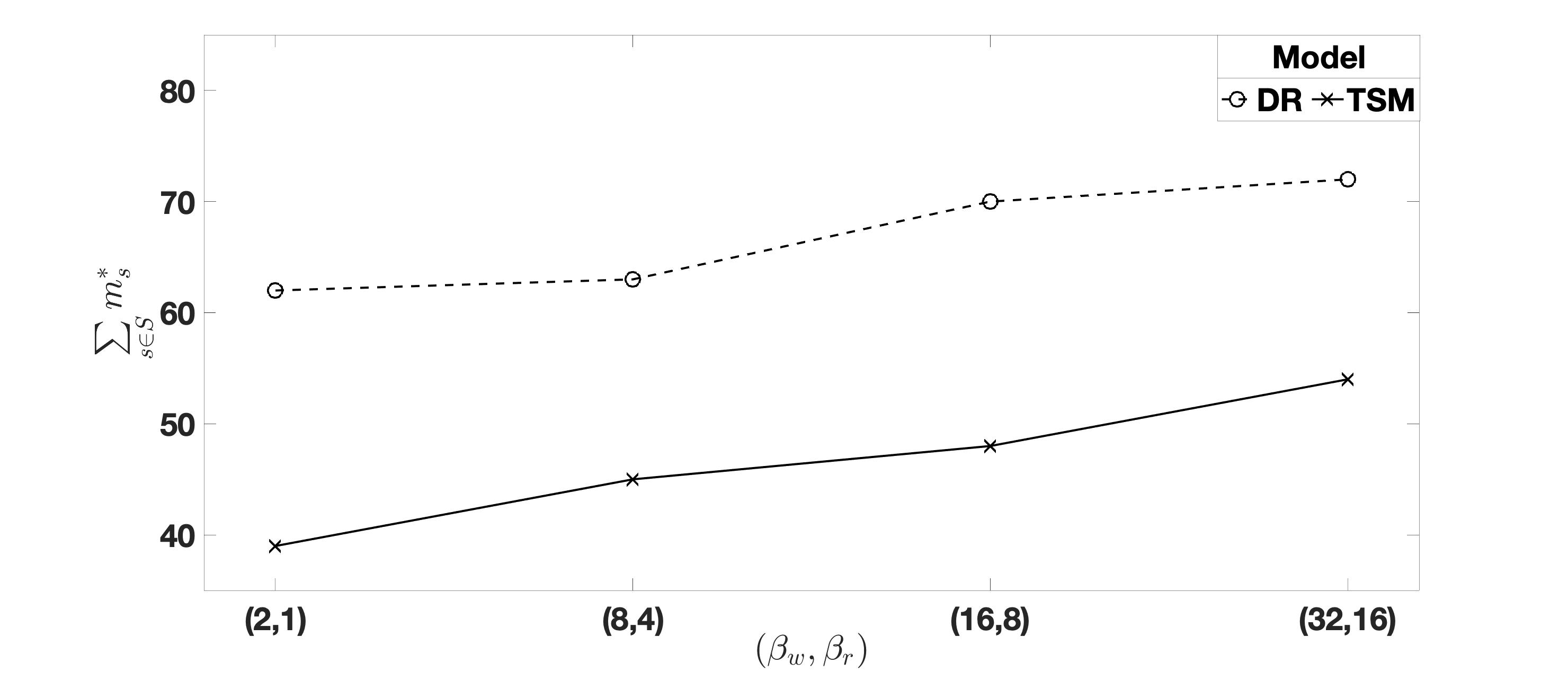}
        \caption{Optimal fleet size vs. ($\beta_w, \beta_r$), $f_s=4,000$}
        \label{Fig1a:M_4000_4_7}
    \end{subfigure}%
    \begin{subfigure}[b]{0.5\textwidth}
            \includegraphics[width=\textwidth]{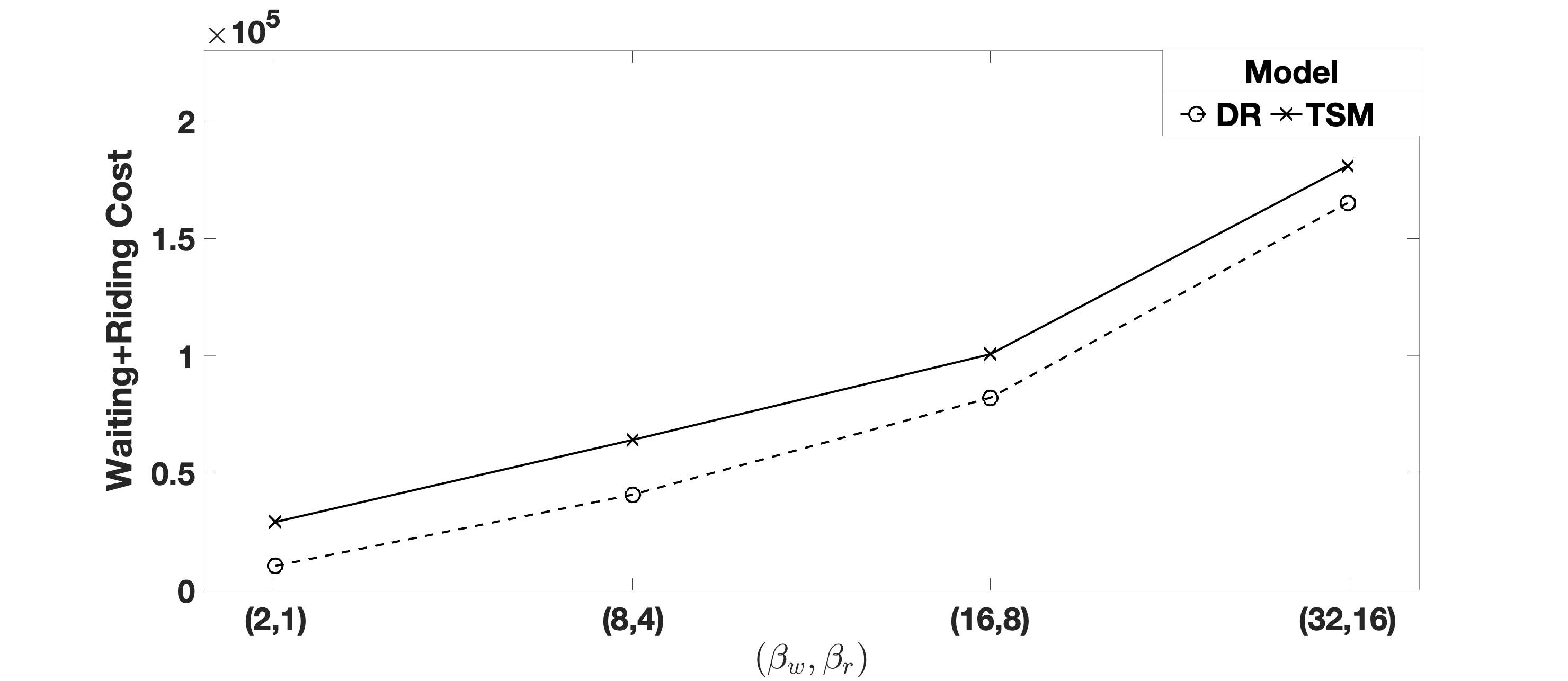}
      \caption{Average 2nd-stage cost vs. ($\beta_w, \beta_r$), $f_s=4,000$}
      \label{Fig1b:Cost_40000_4_7}
    \end{subfigure}%
    
     \begin{subfigure}[b]{0.5\textwidth}
 \centering
        \includegraphics[width=\textwidth]{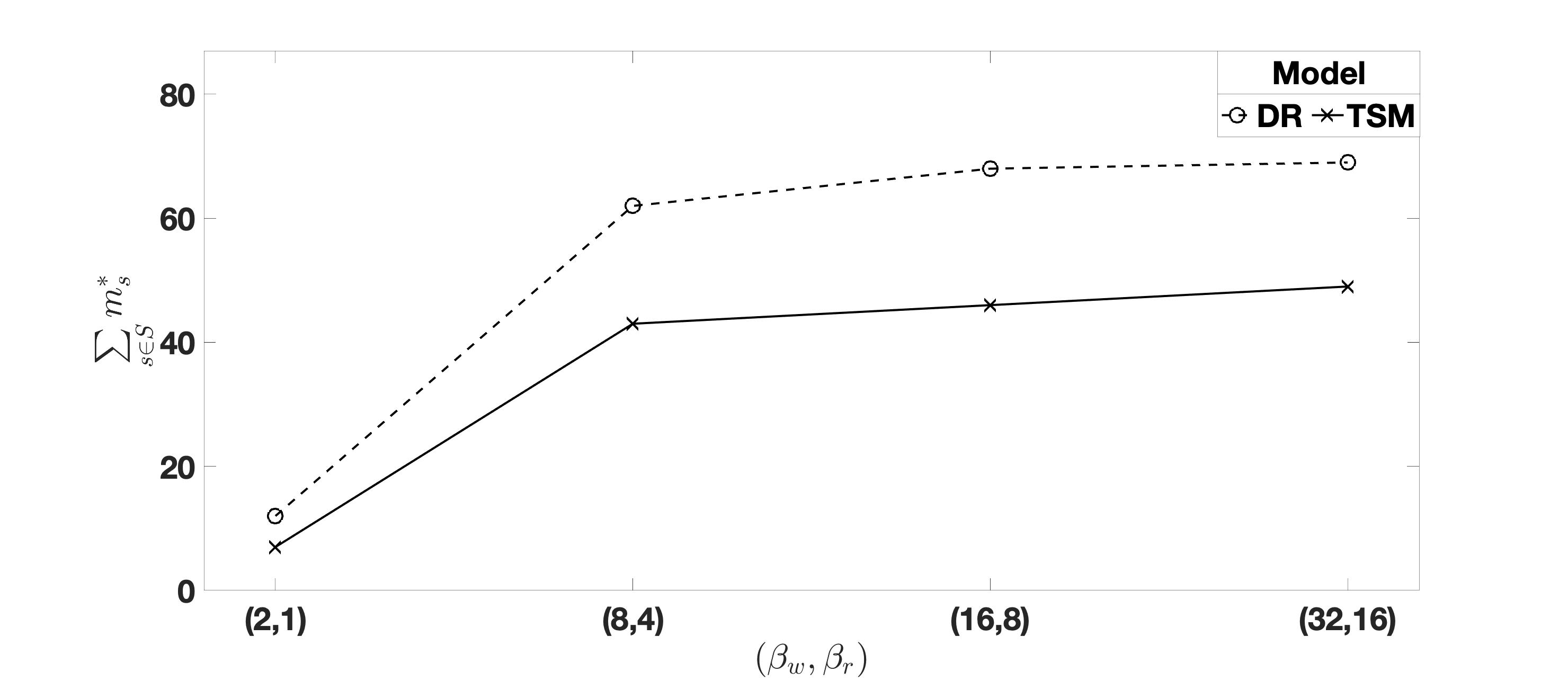}
        \caption{Optimal fleet size vs. ($\beta_w, \beta_r$), $f_s=7,000$}
        \label{Fig1a:M_7000_4_7}
    \end{subfigure}%
    \begin{subfigure}[b]{0.5\textwidth}
            \includegraphics[width=\textwidth]{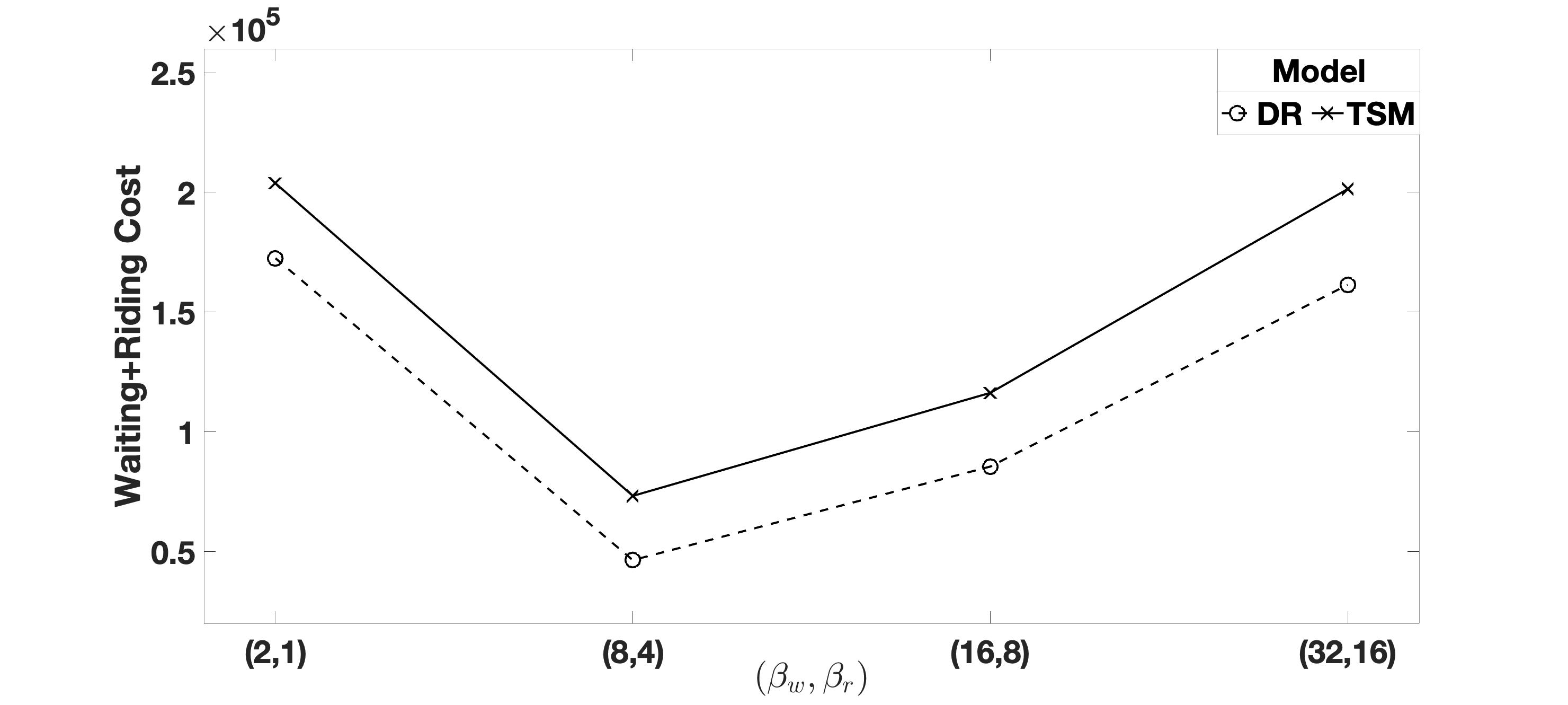}
      \caption{Average 2nd-stage cost vs. ($\beta_w, \beta_r$), $f_s=7,000$}
      \label{Fig1b:Cost_70000_4_7}
    \end{subfigure}%
    
\begin{subfigure}[b]{0.5\textwidth}
 \centering
        \includegraphics[width=\textwidth]{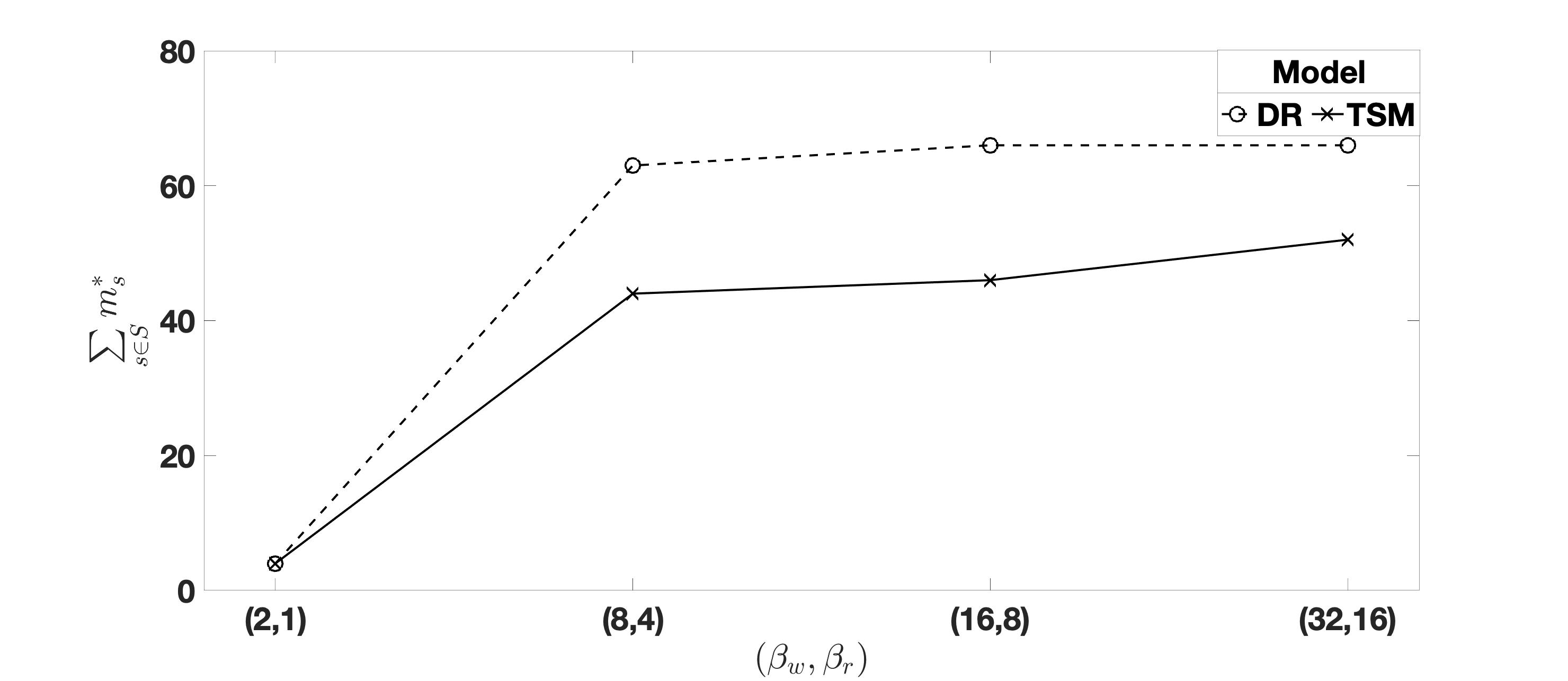}
        \caption{Optimal fleet size vs. ($\beta_w, \beta_r$), $f_s=10,000$}
        \label{Fig1a:M_10000_4_7}
    \end{subfigure}%
    \begin{subfigure}[b]{0.5\textwidth}
            \includegraphics[width=\textwidth]{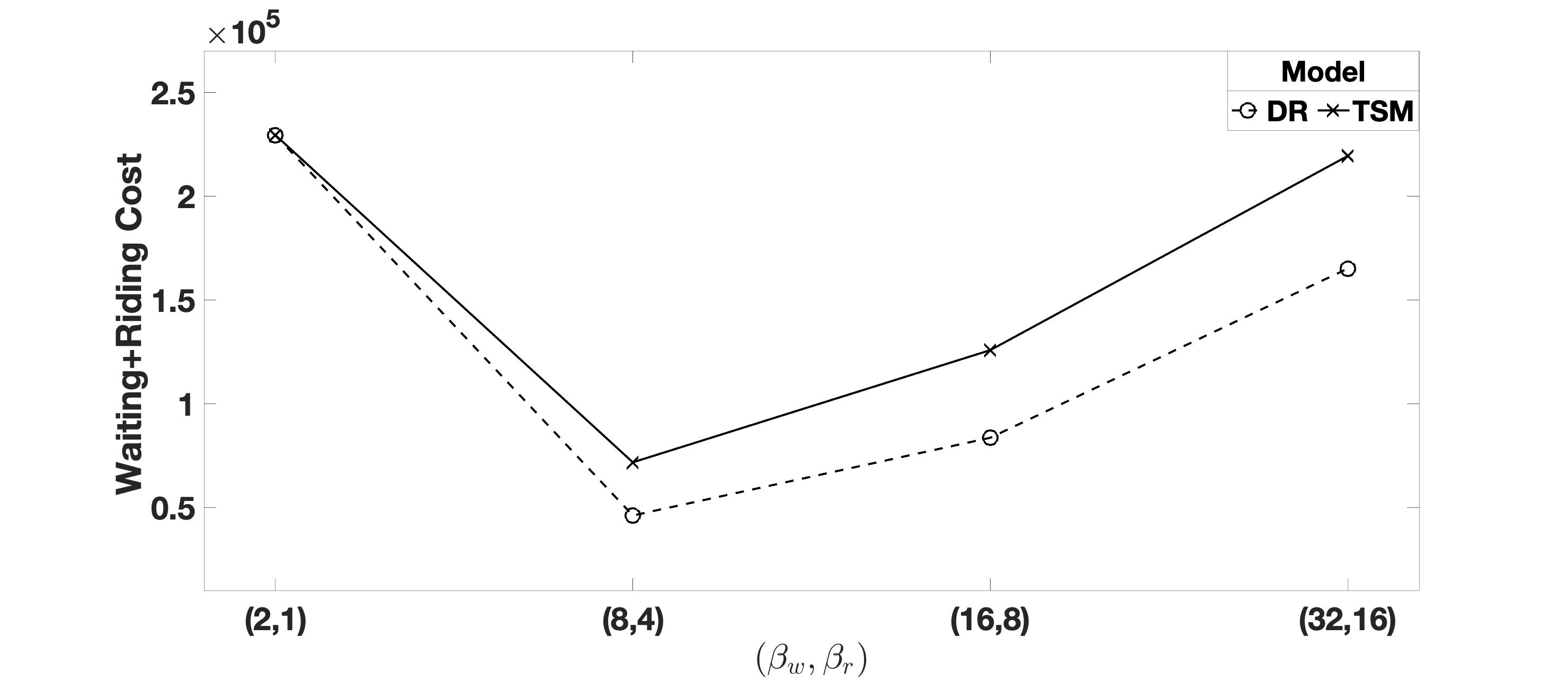}
      \caption{Average 2nd-stage cost vs. ($\beta_w, \beta_r$), $f_s=10,000$}
      \label{Fig1b:Cost_100000_4_7}
    \end{subfigure}%
    
        \caption{Optimal fleet size ($\sum \limits_{s \in S} m_s^*$) and the associated second-stage cost (waiting+riding time costs) under demand range $[4,7]$ and different values of $f_s$ and ($\beta_w, \beta_r$). TSM is the two-stage SP model.}\label{Fig:Sens3}
    \end{figure}

%
%
%
%

\section{Conclusion}\label{sec:Conclusion}

\noindent In this paper, we investigate the fleet sizing and allocation problem for the on-demand last-mile transportation systems. Specifically, we consider the perspective of a last-mile service provider who wants to determine the number of servicing vehicles allocated to multiple service regions. In each service region, passengers demanding last-mile services arrive in batches, and allocated vehicles deliver passengers to their final destinations. The size of each batch of passengers is random and hard to predict in advance. \blue{The quality of fleet-allocation decisions is a function of vehicle fixed cost plus a weighted sum of passenger's waiting time before boarding a vehicle and in-vehicle riding time.} 

We propose, analyze, and evaluate the computational and operational performance of two models for the fleet sizing and allocation problem, assuming known and unknown distribution of the demand, respectively. First, we propose a stochastic programming model to minimize the the fixed cost of allocated vehicles and the expectation of a weighted sum of passenger waiting and riding times, under a distributional belief of demand. Second, we propose a distributionally robust model to minimize the fixed cost of vehicles and the worst-case (i.e., maximum) expectation of passenger waiting time and riding times. We also conduct a set of numerical experiments and discuss the insights and implications by examining trade-offs between total cost, fleet size, and passenger waiting and riding times.

Our study opens other avenues that merit further exploration. To name a few, (1) LMTS fleet sizing and allocation given train arrival uncertainty; (2) LMTS planning and operations under certain special types of demand uncertainty; e.g., multi-modal distribution of demand; (3) fleet sizing and allocation for an on-demand transportation system that combines last- and first-mile services; (4) pricing for last-mile services for multiple service regions with demand uncertainty; \textcolor{black}{(5) incentive and subsidy mechanism design if drivers in the fleet are independent income-seeking decision-makers (e.g., \citet{sun2019model}, \citet{zhu2021mean}); and (6) distributionally robust optimization models for other optimization problems in on-demand transportation; e.g., vehicle allocation, routing, and scheduling for hybrid services with both fixed and flexible routes. Finally, our computational experiment is not all based on real-world case studies or exact data due to the lack of benchmark instances on the specific LMTS problem that we address in this paper. We hope that our approach and results will also motivate future data collection efforts and standardized benchmark instances. The availability of high-quality data will enable the development of data-driven approaches for this and other emerging LMTS problems. }

\vspace{3mm}

\vspace{1mm}
\noindent \textbf{\textbf{\blue{Acknowledgment}}}

\vspace{1mm}

\noindent \blue{We want to thank all of our colleagues and practitioners who have contributed significantly to the related literature.  We are grateful to the anonymous reviewers for their insightful comments and suggestions that allowed us to improve the paper. Dr. Karmel S. Shehadeh dedicates her effort in this paper to every little dreamer in the whole world who has a dream so big and so exciting. Believe in your dreams and do whatever it takes to achieve them--the best is yet to come for you.}

\vspace{2mm}

\noindent \textbf{References}

\bibliographystyle{elsarticle-harv}
\bibliography{LMTSRef}

\newpage

\clearpage 
\appendix
\section{\blue{Proof of Proposition~\ref{Prop2}}}\label{Appex:Proof_Prop2}

\begin{proof}

\begin{subequations}\label{Inner2_Proof}
\begin{align}
\max_{\nb \in [\pmb{\nL,\nU}]} &  \Big \{ \sum \limits_{s \in S} \sum \limits_{i \in I_s} \sum \limits_{j \in J_s}   n_{i,j,s} (\Gamma_{i,j,s}-\rho_{i,j,s}) + \sum \limits_{s \in S} \sum \limits_{i \in I_s} \sum \limits_{k \in K_s} c w_{i,k,s} \psi_{i,k,s}  \Big \} \label{ObjDual_Proof}\\
\st & \ \   \Gamma_{i,j,s}- \Gamma_{i+1,j,s} \leq \beta^{\mbox{\tiny w}} h, \qquad \forall (s, i, j) \label{ConstDual1_proof}\\
 & \ \  \phi_{j,k} (\Gamma_{i,j,s}+ \psi_{i,k,s}) \leq t_{j,k} \beta^{\mbox{\tiny r}}, \qquad \forall (i, j, k, s) \label{ConstDual2_Proof}\\
 & \ \ \psi_{i,k,s} \leq 0, \ \Gamma_{i,j,s} \geq 0, \  \Gamma_{I+1,j,s}=0, \qquad \forall (i, j, k, s)\label{ConstDual3_Proof}
\end{align}
\end{subequations}

\noindent First, we rewrite constraints \eqref{ConstDual2_Proof} as $ \phi_{j,k} \psi_{i,k,s} \leq   t_{j,k} \beta^{\mbox{\tiny r}}- \phi_{j,k} \Gamma_{i,j,s} $. Given that $\psi_{i,k,s} \leq 0$ and the objective of maximizing a positive number times $\Gamma_{i,j,s}$ in \eqref{Inner2_Proof},  then, without loss of optimality, we can assume that $\Gamma_{i,j,s} \geq 0$, for all $i \in [I], \ j\in [J], $ and $s \in [S]$. Note that if $\Gamma_{i,j,s} <0$, then $\phi_{j,k} \psi_{i,k,s} \leq  t_{j,k} \beta^{\mbox{\tiny r}}+\phi_{j,k} |\Gamma_{i,j,s}|=$ a positive number. Given that $\psi_{i,k,s} \leq 0$, then in this case, condition $\phi_{j,k} \psi_{i,k,s} \leq $ a positive number is relaxed and the first term in the objective will be negative for $\Gamma_{i,j,s}$. It follows that $\Gamma_{i,j,s} \geq 0$ in the optimal solution.  \blue{Second, we consider the following  cases, for fixed $s \in S, i \in I_s, j \in J_s$ and fixed $\rho_{i,j,s}$:}

\vspace{2mm}
\begin{itemize}

\item  \textbf{Case 1}: when $\rho_{i,j,s} \geq 0$
\begin{itemize}\itemsep0em
\item If $0 \leq  \Gamma_{i,j,s}<\rho_{i,j,s} $. In this case, $(\Gamma_{i,j,s}-\rho_{i,j,s}) <0$ and so $n_{i,j,s}=\nL_{i,j,s}$ maximizes \eqref{ObjDual_Proof}. Thus, the first term in the objective function \eqref{ObjDual_Proof} reduces to:
\begin{align}\label{eq:inter1}
\nL_{i,j,s} (\Gamma_{i,j,s}-\rho_{i,j,s}) \leq 0
\end{align}
\item If  $ \Gamma_{i,j,s}>\rho_{i,j,s}\geq 0$. In this case, $(\Gamma_{i,j,s}-\rho_{i,j,s}) >0$, $n_{i,j,s}=\nU_{i,j,s}$, and the first term in the objective function \eqref{ObjDual_Proof} reduces to:
\begin{align}\label{eq:inter2}
\nU_{i,j,s} (\Gamma_{i,j,s}-\rho_{i,j,s})  \geq 0
\end{align}

\end{itemize}

Note that \eqref{eq:inter2}$>$\eqref{eq:inter1}. It follows that if $\rho_{i,j,s} \geq 0$, then $ \Gamma_{i,j,s}\geq \rho_{i,j,s}$ and  $n_{i,j,s}=\nU_{i,j,s}$ maximizes the objective of \eqref{Inner2_Proof}. In other words, 
when $\rho_{i,j,s} \geq 0$,  $n_{i,j,s}=\nU_{i,j,s}$ is optimal to \eqref{Inner2_Proof} and the objective equals $\nU_{i,j,s} (\Gamma_{i,j,s}-\rho_{i,j,s})$ as in \eqref{eq:inter2}.

\vspace{2mm}

\item \textbf{Case 2}: when $\rho_{i,j,s} <0 $, then $ \Gamma_{i,j,s}>\rho_{i,j,s}$ (given that $\Gamma_{i,j,s}\geq 0$). In this case, $(\Gamma_{i,j,s}-\rho_{i,j,s})=(\Gamma_{i,j,s}+| \rho_{i,j,s}|) \geq 0$, $n_{i,j,s}=\nU_{i,j,s}$, and the first term in the objective  \eqref{ObjDual_Proof} reduces to:
\begin{align}\label{eq:inter3}
\nU_{i,j,s} (\Gamma_{i,j,s}+| \rho_{i,j,s}|) \geq 0 
\end{align}
Thus,  when $\rho_{i,j,s} <0 $, the objective of \eqref{Inner2_Proof} equals to $\nU_{i,j,s} (\Gamma_{i,j,s}- \rho_{i,j,s}) $

\end{itemize}

\noindent It follows from the above analysis, that $n_{i,j,s}=\nU_{i,j,s}$ is optimal to \eqref{Inner2_Proof}. Thus,, w.l.o.o., we can set $n_{i,j,s}=\nU_{i,j,s}$ in \eqref{Inner2_Proof}. Accordingly, \eqref{Inner2_Proof} is equivalent to 
\begin{subequations}\label{Inner23}
\begin{align}
\max \limits_{\Gamma, \psi} & \ \  \Big \{ \sum \limits_{s \in S} \sum \limits_{i \in I_s} \sum \limits_{j \in J_s} \nU_{i,j,s} (\Gamma_{i,j,s}-\rho_{i,j,s})+ \sum \limits_{s \in S} \sum \limits_{i \in I_s} \sum \limits_{k \in K_s} c w_{i,k,s} \psi_{i,k,s}    \Big\} \\
\ \ \ \st \ & \ \   \Gamma_{i,j,s}- \Gamma_{i+1,j,s} \leq \beta^{\mbox{\tiny w}} h, \qquad \forall (s, i, j) \label{ConstDual13}\\
 & \ \  \phi_{j,k} (\Gamma_{i,j,s}+ \psi_{i,k,s}) \leq t_{j,k} \beta^{\mbox{\tiny r}}, \qquad \forall (i, j, k, s) \label{ConstDual23}\\
 & \ \ \psi_{i,k,s} \leq 0, \ \Gamma_{I+1,j,s}=0, \qquad \forall (i, j, k, s)\label{ConstDual33}
\end{align}
\end{subequations}
For fixed $w$ and $\rho$, problem \eqref{Inner23} is a bounded and feasible linear program. We formulate \eqref{Inner23} in its dual form as 
\begin{subequations}\label{Inner24}
\begin{align}
\min \limits_{y \geq 0,x\geq 0} & \  \sum \limits_{s \in S} \sum \limits_{i \in I_s} \sum \limits_{j \in J_s} \Big[ \beta^{\mbox{\tiny w}} h y_{i,j,s}+ \sum \limits_{k \in K_s} t_{j,k} \beta^{\mbox{\tiny r}} x_{i,j,k,s} \Big]-\sum \limits_{s \in S} \sum \limits_{i \in I_s} \sum \limits_{j \in J_s} \nU_{i,j,s}\rho_{i,j,s} \\
\st & \ y_{0,j,s}+\sum \limits_{k \in K_s} \phi_{j,k} x_{0,j,k,s} \geq  \nU_{0,j,s}, \qquad \forall s \in S,  \ j \in J_s  \\
&  \ y_{i,j,s} -y_{i-1,j,s}+ \sum \limits_{k \in K_s} \phi_{j,k} x_{i,j,k,s} \geq  \nU_{i,j,s}, \qquad \forall s \in S, \  i \in I_s\setminus \{0\}, \  j \in J_s \\
&  \  \sum \limits_{j \in J_s} \phi_{j,k} x_{i,j,k,s}  \leq cw_{i,k,s}, \qquad \forall s \in S, \ i \in I_s, \ j \in J_s
\end{align}
\end{subequations}

\noindent Combining the inner problem in the form of \eqref{Inner24} with the outer minimization problems in  \eqref{eq:FinalDualInnerMax-1} and \eqref{Obj:DRObjective2}, we derive the following equivalent reformulation of the DR model in \eqref{Obj:DRObjective2}
\begin{subequations}
\begin{align}
 \min  &\Bigg \{ \sum_{s \in S} f_s m_s+ \sum \limits_{s \in S} \sum \limits_{i \in I_s} \sum \limits_{j \in J_s}( \mu_{i,j,s}- \nU_{i,j,s} )\rho_{i,j,s}+   \sum \limits_{s \in S} \sum \limits_{i \in I_s} \sum \limits_{j \in J_s} \Big[ \beta^{\mbox{\tiny w}} h y_{i,j,s}+ \sum \limits_{k \in K_s} t_{j,k} \beta^{\mbox{\tiny r}} x_{i,j,k,s} \Big] \Bigg \}\\
 \st & \ y_{0,j,s}+\sum \limits_{k \in K_s} \phi_{j,k} x_{0,j,k,s} \geq  \nU_{0,j,s}, \qquad \forall s \in S,  \ j \in J_s  \\
&  \ y_{i,j,s} -y_{i-1,j,s}+ \sum \limits_{k \in K_s} \phi_{j,k} x_{i,j,k,s} \geq   \nU_{i,j,s}, \qquad \forall s \in S, \  i \in I_s\setminus \{0\}, \  j \in J_s \\
&  \  \sum \limits_{j \in J_s} \phi_{j,k} x_{i,j,k,s}  \leq cw_{i,k,s}, \qquad \forall s \in S, \ i \in I_s, \ j \in J_s\\
& \ m \in \mathcal{M}, \ w \in \mathcal{W}, \  y \geq 0, \ x\geq 0, \ \rho \geq 0
\end{align}
\end{subequations}

\end{proof}

\newpage
\section{Sample average approximation of \blue{SP}}\label{Appx:SAA}
\noindent There are two well-known difficulties in obtaining an (exact) optimal solution to the \blue{SP} in \eqref{SP_model}. First, evaluating the value of  $\mathbb{E}_\mathbb{P} [Q (m, \xi)]$ involves taking multi-dimensional integrals and solving a huge number of similar integer programs. Second, both $\mathbb{E}_\mathbb{P} [Q (x, \xi)]$ and $Q(x, \xi)$ are non-convex and discontinuous \citep{birge2011introduction, shapiro2009lectures}. In view of these two difficulties, we resort to approximation solution approaches, and the sample average approximation (SAA) approach in particular.

\vspace{1mm}

 In SAA, we replace the distribution of $\xi$ with a (discrete) empirical distribution based on $R$ independent and identically distributed (i.i.d.) samples of random demand, then solve the sample average approximation \eqref{SAAModel} of \eqref{SP_model}. Note that in the SAA formulations \eqref{SAAModel}, we associate all scenario-dependent parameters, variables, and constraints with a scenario index $r$ for all $r=1,\ldots, R$. \textcolor{black}{For example, parameters $n$ by $n^{\mbox{\tiny r}}$ to represent  number of passengers realized in scenario $r$, and variables $u$ are replaced by $u^r$ to represent the number of unserved passengers in scenario $r$.} In addition, constraints \eqref{Const1}--\eqref{Const6} are incorporated in each scenario.
\begin{subequations}\label{SAAModel}
\begin{align}
& v_R= \min \hat{f}_R:= \frac{1}{R}  \sum \limits_{r=1}^R \Big[  \ \beta^{\mbox{\tiny w}} \sum \limits_{s\in S}\sum \limits_{i \in I_s} \sum \limits_{j \in J_s}  h u_{i,j,s}^r+ \beta^{\mbox{\tiny r}} \sum \limits_{s\in S} \sum \limits_{i \in I_s} \sum \limits_{j \in J_s} \sum \limits_{k \in K_s} t_{j,k} z_{i,j,k,s}^r \Big]
  \label{SAAObj}\\
&  \ \  \text{s.t.} \ \ \  \ m \in \mathcal{M}, \ w \in \mathcal{W}\label{SAAConst1}\\
&\quad  \   \quad  \quad  \eqref{Const1}-\eqref{Const6}, \ \ \text{for } r=1,\ldots,R
\end{align}
\end{subequations}
 
\newpage

\color{black}
\section{\blue{Construction and Statistics of the NYC Instance}}\label{AppexNYC}
\blue{We construct the NYC instance based on the dataset the procedure as follows:}
\color{black}
\begin{enumerate}\itemsep0em
    \item Select 4 metro stations that are relatively far away from each other in Manhattan, NYC;
    \item For each station, construct a 1-mile by 1-mile square as a last-mile service region with the station located in the center;
    \item For each service region, consider the passengers with destination within the region as potential demand for LMTS;
    \item For each service region $s$, cluster the passenger destinations to $J_s$ clusters; assume the center of each cluster $j_s$ as the location of a last-mile stop $j_s$ to cover all passengers in that cluster;
    \item For each last-mile stop $j$ in service region $s$, record the number of passengers going to that stop (i.e., with destination in that cluster) within each time interval $i$ (e.g., every 5 or 10 minutes) as the batch demand $n_{i,j,s}$ for the LMTS; 
    \item For each last-mile stop $j$ in service region $s$, compute the upper bound, lower bound, $20\%$ percentile, and $80\%$ percentile for batch demand $n_{i,j,s}$ across $i$ in certain period (e.g., 10 am to 11 am);
    \item  Using the locations of all last-mile stops in each service region $s$, generate $K_s$ routes to serve a subset of stops (e.g., serve 1, 2, and 3 stops), all of which start from and return to the metro station; 
    \item For each route $k$, record its total travel time $t_k$, stop-route configuration $\phi_{j,k}$, and travel time to each stop $t_{j,k}$.
\end{enumerate}

\setcounter{table}{0}
\begin{table}[h!]
 \center 
 \color{black}
   \renewcommand{\arraystretch}{0.3}
\noindent  \caption{\blue{NYC Instance. \\Notation: $S$ is number of regions, $s$ is a region, $J_s$ is number of last-mile stops in regions $s$, $K_s$ is number of routes in region $s$.}}
\begin{tabular}{ccccccccccccc}
\hline
\textbf{Inst} & \textbf{$S$} & $s$ &   \textbf{$J_s$}   & \textbf{$K_s$} \\
\hline
NYC &  4   & $ \ \ s=1$ &  5& 31  \\
&  & $ \ \ s=2$ &  6 & 30 \\
&& $ \ \ s=3$ &  4& 15\\
&& $ \ \ s=4$ &  5&20\\
\hline
\end{tabular} 
\label{table:NYC}
\end{table}


 \begin{table}[h!]
 \center 
 \color{black}
   \renewcommand{\arraystretch}{0.4}
  \caption{\blue{Statistics of the batch demand for each last-mile stop per service region in the NYC instance. \\Notation:$\mu$ and $\sigma$ are respectively the empirical mean and standard deviations of batch demand $n_{i,j,s}$}.}
\begin{tabular}{llllllll}
\hline
Region, $s$	& LM stop, $j$ &  $\mu $	&		$\sigma$ \\
\hline
$s=1$ & $j=1$ & 1.11 & 1.64 \\
 & $j=2$ &1.13 & 1.88\\
 & $j=3$ &2.23 & 2.72\\
 &$j=4$ & 1.13& 2.45\\
  &$j=5$ & 1.30& 2.25\\
 \\
 $s=2$ & $j=1$ & 1.40& 1.72 \\
 & $j=2$ & 1.43 & 1.93 \\
 & $j=3$ & 1.20 & 1.92\\
 & $j=4$ &2.63 & 2.54\\
  & $j=5$ &2.63 & 2.54\\
  & $j=6$ &1.83 & 1.89\\
  \\
 $s=3$ & $j=1$ &2.43& 2.86 \\
 & $j=2$ & 3.33 & 2.37\\
 & $j=3$ &2.63 & 3.10 \\
& $j=4$ & 3.23 & 3.58\\
\\
$s=4$ &  $j=1$ & 1.37 &2.86\\
 & $j=2$ &  1.13& 2.37\\
 & $j=3$ &1.90 & 3.10\\
& $j=4$ &1.20 & 3.58\\
  & $j=5$ &1.70 & 2.58\\
\hline	
\end{tabular} 
\label{table:NYCStats}
\end{table}

\newpage
\color{black}

\section{Values of parameters $f_s$, $\beta^w$, and $\beta^r$}\label{Appix:fixed cost}

\noindent Let the average after-tax hourly wage of passengers be $g$/hour. According to \cite{gomez1999essays}:
\begin{itemize}\itemsep0em
\item Monetary value of riding time ($76\%$ of after-tax wage): $\$0.76g$/hour$ \approx\$0.0127g$/minute.
\item Monetary value of waiting time ($195\%$ of after-tax wage): $\$1.95g$/hour=$\$0.0325g$/minute.
\item Average hourly total fixed cost of a vehicle (with capacity 5), including the cost to rent the vehicle, wage paid to the driver, fuel cost, and other maintenance and operating costs: $b$/hour.

\end{itemize}

In the LMTS, we have $I$ trains with headway $h$ minute. The duration of the time that vehicles are needed is slightly longer than $Ih$. Then, we can approximate the fixed cost $f_s$, parameter $\beta^w$, and $\beta^r$ as $f_s=b\cdot\frac{Ih}{60}=\frac{bIh}{60}$; $\beta^r=0.0127g$; $\beta^w=0.0325g$.

Scenario 1: Assuming there is an existing fleet with no additional cost, we have $f_s=0$, $\beta^r=0.0127g$, and $\beta^w=0.0325g$. In the numerical experiments, we normalize $\beta^r$ to 1 and evaluate this scenario with the following parameters: $f_s=0$, $\beta^r=1$; $\beta^w\in[2,3]$.

Scenario 2: Assuming the average after-tax hourly wage of passengers (e.g., working adults who are more sensitive to riding and waiting times) $g=\$10$/hour and the average hourly total cost of a vehicle (e.g., regular vehicle) $b=\$30$/hour, then:
$f_s=\frac{30\cdot 10 \cdot 10}{60}=50$; $\beta^r=0.127$; $\beta^w=0.325$. In the numerical experiments, we normalize $\beta^r$ to 1 and evaluate this scenario with the following parameters: $f_s \in[200,600]$, $\beta^r=1$; $\beta^w\in[2,3]$.

Scenario 3: Assuming the equivalent average after-tax hourly wage of passengers (e.g., children, students, seniors, and the disabled, who are less sensitive to riding and waiting times) $g=\$3$/hour and the average hourly total cost of a vehicle (e.g., vehicle with special equipment for children, seniors, and the disables) $b=\$90$/hour, then
$f_s=\frac{90\cdot 10 \cdot 10}{60}=150$; $\beta^r=0.0381$; $\beta^w=0.0975$. In the numerical experiments, we normalize $\beta^r$ to 1 and evaluate this scenario with the following parameters: $f_s \in[2,000,6,000]$, $\beta^r=1$; $\beta^w\in[2,3]$.


\setcounter{table}{0}
\section{\textcolor{black}{Example of In-sample performance under $f_s=0$}}\label{Appex:InSampleZeroF}

\noindent \blue{Under $f_s=0$, both models allocate the same numbers of vehicles. As such, they have similar in-sample and out-of-sample simulation performances for all three instances.  As an example, in Table  ~\ref{table:InSamplef=0}, we present in-sample simulation results (i.e., under set 1; perfect information) for instance 1.}

\begin{table}[h!!]
 \footnotesize

  \centering
\caption{In-sample performance of optimal allocation decisions under perfect distributional information, $f_s=0$. The performance of two models are almost the same if demands follow the perfect (in-sample) distribution.} 
    \begin{tabular}{rrrlrr}
    \toprule
    \multicolumn{1}{l}{\textbf{Instance}} & \multicolumn{1}{l}{\textbf{R}} & \multicolumn{1}{l}{\textbf{Metric}} & \textbf{Model} & \multicolumn{1}{l}{\textbf{TC}} & \multicolumn{1}{l}{\textbf{2nd-Stage}} \\
    \midrule
    1     & \multicolumn{1}{l}{[1,4]} & \multicolumn{1}{l}{Mean} & DR    & 4,590  & 4,590 \\
          &       &       & \blue{SP}   & 4,572  & 4,572 \\
          &       & \multicolumn{1}{l}{Median} & DR    & 4,592  & 4,592 \\
          &       &       & \blue{SP}   & 4,573  & 4,573 \\
          &       & \multicolumn{1}{l}{75\%-q} & DR    & 4,722  & 4,722 \\
          &       &       & \blue{SP}   & 4,750  & 4,750 \\
          &       & \multicolumn{1}{l}{95\%-q} & DR    & 4,921  & 4,921 \\
          &       &       & \blue{SP}   & 5,001  & 5,001 \\
          & \multicolumn{1}{l}{[3,7]} & \multicolumn{1}{l}{Mean} & DR    & 9,370  & 9,370 \\
          &       &       & \blue{SP}   & 9,618  & 9,618 \\
          &       & \multicolumn{1}{l}{Median} & DR    & 9,332  & 9,332 \\
          &       &       & \blue{SP}   & 9,605  & 9,605 \\
          &       & \multicolumn{1}{l}{75\%-q} & DR    & 9,776  & 9,776 \\
          &       &       & \blue{SP}   & 9,941  & 9,941 \\
          &       & \multicolumn{1}{l}{95\%-q} & DR    & 10,403 & 10,403 \\
          &       &       & \blue{SP}   & 10,509 & 10,509 \\
    \bottomrule
    \end{tabular}%
\label{table:InSamplef=0}
\end{table}%

\end{document}